\documentclass[11pt]{amsart}
\usepackage{amssymb}
\usepackage{amsmath}
\usepackage{amsfonts}
\usepackage{amsthm}

\theoremstyle{plain} \numberwithin{equation}{section}
\newtheorem{Theorem}{Theorem}
\newtheorem{Lemma}[Theorem]{Lemma}
\newtheorem{Proposition}[Theorem]{Proposition}
\newtheorem{Corollary}[Theorem]{Corollary}
\newtheorem{Remark}[Theorem]{Remark}

\begin{document}

\title[Equiconvergence]
{Equiconvergence of spectral decompositions of  Hill-Schr\"odinger
operators}
\author{Plamen Djakov}

\author{Boris Mityagin}

\address{Sabanci University, Orhanli,
34956 Tuzla, Istanbul, Turkey}
 \email{djakov@sabanciuniv.edu}
\address{Department of Mathematics,
The Ohio State University,
 231 West 18th Ave,
Columbus, OH 43210, USA} \email{mityagin.1@osu.edu}

\begin{abstract}
We study in various functional spaces the equiconvergence of
spectral decompositions of the Hill operator $L= -d^2/dx^2 + v(x), $
$x \in [0,\pi], $ with $H_{per}^{-1} $-potential and the free
operator $L^0=-d^2/dx^2, $ subject to periodic, antiperiodic or
Dirichlet boundary conditions.

In particular, we prove that
$$  \|S_N - S_N^0: L^a \to L^b \| \to 0 \quad
\text{if} \;\; 1<a \leq b< \infty, \;\; 1/a - 1/b <1/2,  $$ where
$S_N$ and $S_N^0 $ are the $N$-th partial sums of the spectral
decompositions of $L$ and $L^0.$   Moreover, if $v \in H^{-\alpha} $
with $1/2 < \alpha < 1$ and $\frac{1}{a}=\frac{3}{2}-\alpha, $ then
we obtain uniform equiconvergence: $\|S_N - S_N^0: L^a \to L^\infty
\| \to 0 $ as $N \to \infty. $
\vspace{1mm} \vspace{1mm} 
\end{abstract}

\maketitle

\section*{Content}
\begin{enumerate}

\item[Section 1.]  Introduction   \hspace{3mm}.
.   .   .   .   .   .   .   .   .   .   .   .   .   .   .   .   .   .
.   .   .   .   .   .   .   .   .   .   .   .   1 \vspace{3mm}

\item[Section 2.] The case of potentials $v \in L^p, \, p\in(1,2] $\hspace{4mm}
    .   .   .   .   .   .   .   .   .   .   .   .   .   .    9
   \vspace{3mm}

\item[Section 3.] $ L^1$-potentials and weakly singular potentials
 \hspace{4mm}   .   .   .   .   .   .   .   .   .   .   13 \vspace{3mm}

\item[Section 4.] The case of potentials $v\in H^{-\alpha}, \,
0 < \alpha<1$\hspace{3mm}   .   .   .   .   .   .   .   .   .   .   .
. 18
 \vspace{3mm}

\item[Section 5.] The case $v \in H^{-1}_{per},$
 $S_N -S_N^0:L^a \to L^b, \; b <\infty$ \hspace{3mm}
 .   .   .   .   .   .   .   .   .   35 \vspace{3mm}

\item[Section 6.]  Appendix: Auxiliary Inequalities
\hspace{4mm}   .   .   .   .   .   .   .   .   .   .    .   .   .
    .   .   .   .   42 \vspace{3mm}

\item[]   References \hspace{2mm}   .   .   .   .   .   .   .   .   .   .
.   .   .   .   .   .   .   .   .   .   .   .   .   .   .   .   .   .
.   .   .   47

\end{enumerate}
\bigskip

\section{Introduction}

  1. Since the earlier days of the theory of eigenfunction expansions
for ordinary differential operators (V. A. Steklov \cite{Ste07,
Ste10}, G. D. Birkhoff \cite{1-Bi08, 2-Bi08}, A. Haar \cite{Ha10})
one of a few central questions was the question about {\em
equiconvergence} of eigenfunction expansions related to the same
ordinary differential operator (o.d.o.) but subject to different
Birkhoff-regular boundary conditions \cite{Na69} or for o.d.o. with
different coefficients but the same (or similar) boundary
conditions.
   To illustrate the problem let us remind two results of
J. Tamarkin \cite{Ta12,Ta17,Ta28}.

Let
\begin{equation}
\label{1} l(y) = \frac{d^n y}{dx^n} + \sum_{k=0}^{n-2} p_k(x)
y^{k}(x),
  \quad 0 \leq x \leq 1, \quad p_k \in L^1 ([0, 1]),
  \end{equation}
and $n$ linearly independent $bc$ (boundary conditions) which are
{\em regular} (see \cite{Na69}) define an operator $L$ in
$L^2([0,1])$. Let
  $\Lambda = \{\lambda_j\}_1^\infty $
  be the set of all eigenvalues of $L$, and
$R(z) = (z - L)^{-1}$ be its resolvent. Define

\begin{equation}
\label{2} S_r(f) = \frac{1}{2 \pi i} \int_{C(r)} R(z) dz,
\end{equation}
  where
$$C(r) = \{z \in \mathbb{C} : |z| = r\}$$
with  radii $r$ chosen in such a way that
\begin{equation}
\label{3} \text{dist} (C(r), \Lambda) \geq \varepsilon > 0,
\end{equation}
and define the $r$-th "partial sum" of the trigonometric Fourier
integral
\begin{equation}
\label{4}  \sigma_r (f) = \frac{1}{\pi} \int_\mathbb{R} \frac{\sin
r(x-\xi)}{x-\xi}
  f(\xi) d\xi, \quad   f \in L^1([0,1]).
\end{equation}

  {\bf Claim 1} (J. Tamarkin \cite{Ta17, Ta28},
 M. Stone \cite{Sto26, Sto27}).
{\em With notations} (\ref{2})-(\ref{4})
   {\em the following holds:
\begin{equation}
\label{5} \lim_{r \to \infty, \,r \in (\ref{3})} \| S_r(f) -
\sigma_r(f)\|_{C(K)} = 0
\end{equation}
for any compact} $K$ in $(0, 1).$

   {\bf Claim 2} (J. Tamarkin \cite{Ta17, Ta28}).
{\em If $L^0$ is the free operator  $ \frac{d^n y}{dx^n}$, i.e., $p_k
= 0,$ then}
\begin{equation}
\label{6}\lim_{r \to \infty, \,r \in (\ref{3})} \| S_r(f, L) -
S_r(f, L^0)\|_{C[0,1]} = 0 \quad \forall f \in L^1([0,1]).
\end{equation}

   These two statements bring our attention to the distinction between
equiconvergence {\em on compacts} inside of the open interval $(0,
1)$ (Claim 1) and {\em on the entire closed interval} $[0, 1]$ (Claim
2). Along the first line of research lately let us mention works of
A. P. Khromov \cite{Kh62, Kh81, 1-Kh95, 2-Kh95},
 A. Minkin \cite{Mi99}, V. S. Rykhlov \cite{Ry80, Ry83, Ry84, Ry86,
Ry90, Ry96}, A. M. Gomilko and G. V. Radzievskii \cite{GR91}, A. S.
Lomov \cite{Lo05}.

Quite exceptional is the paper \cite{Kh75} where a criterion of
equiconvergence on the whole interval for two different
Birkhoff-regular bvp and a given continuous function was found.
\bigskip

2. In the present paper we will focus on {\em (equi)convergence on
the whole interval} in the case of o.d.o. of the second order, or
Hill operators
\begin{equation}
\label{7} L = -  \frac{d^2 y}{dx^2} + v(x), \quad 0 \leq x \leq \pi,
\end{equation}

with $bc$ of three types:

(a) {\em periodic} $Per^+: \quad y(0) = y (\pi), \;\; y^\prime (0) =
y^\prime (\pi) $;

(b) {\em anti--periodic}  $Per^-: \quad  y(0) = - y (\pi), \;\;
y^\prime (0) = - y^\prime (\pi) $;

(c) {\em Dirichlet} $ Dir: \quad y(0) = 0, \;\; y(\pi) = 0. $

By using the quasi--derivatives approach of
 A. Savchuk and A. Shkalikov  \cite{SS00,SS03} (see
also \cite{SS01,SS05,SS06}  and
 R. Hryniv and Ya. Mykytyuk \cite{HM01}--\cite{HM066}),
 we developed in \cite{DM17,DM16,DM21} a
Fourier method for studying the spectral properties of one
dimensional Schr\"odinger operators with periodic complex-valued
singular potentials of the form
\begin{equation}
\label{001} v  = Q^\prime, \quad Q\in L_{loc}^2 (\mathbb{R}), \quad
Q(x+\pi) =Q (x).
\end{equation}
Following A. Savchuk and A. Shkalikov \cite{SS00,SS03}, one may
consider various boundary value problems on the interval $[0,\pi])$
in terms of quasi-derivative
$$
y^{[1]} = y^\prime  - Qy.
$$
In particular, the periodic and anti--periodic boundary conditions
have the form

 $\qquad   Per^+: \quad y(\pi)= y(0), \quad y^{[1]}(\pi)=
y^{[1]}(0), $

  $\qquad  Per^-: \quad y(\pi)= -y(0), \quad y^{[1]} (\pi)= -
y^{[1]} (0). $

Of course, if $Q$ is a continuous function, then $Per^+ $ and
$Per^-$ coincide, respectively, with the classical periodic boundary
condition (a) and (b). The Dirichlet boundary condition has the same
form as in the classical case:

  $\qquad  Dir: \quad y(\pi)= y(0)=0.  $

For each of the boundary conditions $bc= Per^\pm,\, Dir $ the
differential expression
$$
\ell (y) = -(y^{[1]})^\prime -Qy
$$
gives a rise of a closed (self adjoint for real $v$) operator
$L_{bc}=L_{bc}(v)  $ in $H^0 = L^2 ([0,\pi]), $ respectively, with a
domain
\begin{equation}
\label{002} D (L_{Per^\pm}) =\{y \in H^1: \;y^{[1]} \in W^1_1
([0,\pi]), \; Per^\pm \; \text{holds},  \; \ell (y) \in H^0 \},
\end{equation}
or
\begin{equation}
\label{003} D (L_{Dir}) =\{y \in H^1: \; y^{[1]} \in W^1_1 ([0,\pi]),
\; Dir \; \text{holds},  \; \ell (y) \in H^0 \}.
\end{equation}

Let $L^0_{bc}$ denote the free operator $L^0 = - d^2/dx^2 $
considered with boundary conditions $bc.$ It is easy to describe the
spectra and eigenfunctions of $L^0_{bc}$ for $bc = Per^\pm, Dir :$

(a) $ Sp (L^0_{Per^+}) = \{n^2, \; n = 0,2,4, \ldots \};$ its
eigenspaces are  $E^0_n = Span \{e^{\pm inx} \} $ for $n>0 $ and
$E^0_0 = \{ const\}, \; \; \dim E^0_n = 2 $ for $n>0, $ and $\dim
E^0_0 = 1. $

(b) $ Sp (L^0_{Per^-}) = \{n^2, \; n = 1,3,5, \ldots \};$ its
eigenspaces are  $E^0_n = Span \{e^{\pm inx} \}, $ and $ \dim E^0_n
= 2. $

(c)  $ Sp (L^0_{Dir}) = \{n^2, \; n \in \mathbb{N}\};$  its
eigenspaces are $E^0_n = Span \{\sin nx \}, $ and $ \dim E^0_n = 1. $

 Depending on the boundary conditions,
we consider as our canonical orthogonal normalized basis (o.n.b.) in
$ L^2 ([0,\pi]) $ the system $u_k (x), \;k \in \Gamma_{bc}, $ where
\begin{eqnarray}
\label{0.011}
 \mbox{if} \; \; bc =Per^+ & \quad u_k =\exp (ikx),& \;
k \in \Gamma_{Per^+} = 2\mathbb{Z};  \\
\label{0.012} \mbox{if} \; \; bc =Per^- &\quad u_k =\exp (ikx), & \;
k \in \Gamma_{Per^-} = 1+2\mathbb{Z};\\
\label{0.013}
 \mbox{if} \; \; bc =Dir & \quad u_k =\sqrt{2} \sin kx,&
k \in \Gamma_{Dir} = \mathbb{N}.
\end{eqnarray}
Let us notice that  $\{u_k (x), \;k \in \Gamma_{bc}\} $ is a
complete system of unit eigenvectors of the operator $L^0_{bc}.$
They are uniformly bounded, namely
\begin{equation}
\label{0.051} |u_k (x)| \leq \sqrt{2} \quad    \forall k \in
\Gamma_{bc}.
\end{equation}

We set
\begin{equation}
\label{006} H^1_{Per^\pm} = \left \{f \in H^1: \; f(\pi) = \pm f(0)
\right \}, \quad H^1_{Dir} = \left \{f \in H^1: \; f(\pi) = f(0)=0
\right \}.
\end{equation}
One can easily see that $\{e^{ikx}, \; k \in \Gamma_{Per^\pm} \} $
is an orthogonal basis in $H^1_{Per^\pm}$ and $\{\sqrt{2} \sin kx,
\; k \in \mathbb{N} \} $ is an orthogonal basis in $H^1_{Dir}.$ From
here it follows that
\begin{equation}
\label{009} H^1_{bc} = \left \{ f(x) = \sum_{k\in \Gamma_{bc}} f_k
u_k (x) : \; \|f\|^2_{H^1} =\sum_{k\in \Gamma_{bc}} (1+k^2)|f_k|^2 <
\infty \right \}.
\end{equation}

The following proposition gives the Fourier representation of the
operators $L_{Per^\pm}$ and their domains (see
\cite[Prop.10]{DM16}). Let $v$ be a singular potential of the form
(\ref{001}) and let $Q(x) = \sum_{k\in 2\mathbb{Z}} q(k)e^{ikx}
 $ be the Fourier series of $Q$ with respect to the orthonormal
system $\{e^{ikx}, \; k\in 2\mathbb{Z}\}.$  We set
\begin{equation}
\label{0010} V(k)= i \, k \cdot q(k), \quad   k\in 2\mathbb{Z}.
\end{equation}

\begin{Proposition}
\label{prop001} In the above notations, if $y \in H^1_{Per^\pm}, $
then we have $y=\sum_{\Gamma_{Per^\pm}} y_k e^{ikx} \in
D(L_{Per^\pm})$  and  $ L  y =  h = \sum_{\Gamma_{Per^\pm}} h_k
e^{ikx}  \in H^0$ if and only if
\begin{equation}
\label{0011}  h_k = h_k (y) : =  k^2 y_k + \sum_{m\in
\Gamma_{Per\pm}} V(k-m) y_m,  \quad   \sum |h_k|^2 < \infty,
\end{equation}
i.e.,
\begin{equation}
\label{0012} D(L_{Per^\pm}) = \left \{y \in H^1_{Per^\pm}:  \quad
(h_k (y))_{k \in  \Gamma_{Per^\pm}} \in \ell^2 \left
(\Gamma_{Per^\pm} \right ) \right \}
\end{equation}
and
\begin{equation}
\label{0013} L_{Per^\pm} (y)  = \sum_{k \in \Gamma_{Per^\pm}} h_k
(y) e^{ikx}.
\end{equation}

\end{Proposition}

In the case of Dirichlet boundary conditions, we consider expansions
about the o.n.b. $\{\sqrt{2}\sin kx, \, k \in \mathbb{N}\}.$ Let
\begin{equation}
\label{d.1} Q(x) =   \sum_{k=1}^\infty \tilde{q} (k) \sqrt{2} \sin kx
\end{equation}
be the sine Fourier expansion of $ Q. $ We set
\begin{equation}
\label{d.3} \tilde{V} (0) = 0, \qquad \tilde{V}(k) = k \tilde{q} (k)
\quad  \mbox{for} \; k \in \mathbb{N}.
\end{equation}
\begin{Remark}
\label{remDir} Since $v=Q^\prime,$ the function $Q(x) $ is defined up
to a constant. The choice of this constant play no role in the case
of periodic or antiperiodic boundary conditions -- the coefficients
$V(k) $ in (\ref{0010}) do not depend on such a choice. But in the
case of Dirichlet boundary conditions the situation is different.
Since $$ \frac{1}{\pi} \int_0^\pi \sin mx \, dx = \begin{cases}  0  &
\text{for even} \; \; m,\\
2/m  & \text{for odd} \; \; m,  \end{cases}
$$
if one add a constant $C $ to $Q(x) $ then for odd $m$ the
coefficients $\tilde{V} (m) $ will change by $ 2C. $

If $v\in L^1([0,\pi])$ and $\int_0^\pi v(x) dx=0,$ then with $Q(x)
=\int_0^x v(x) dx $ it follows that the numbers $\tilde{V}(k) $ are
the Fourier coefficients of $v$ about the o.n.b. $\{\sqrt{2} \cos kx,
\, k \in \mathbb{Z}_+\}.$  The following choice of constants
guarantees that our formulas agree with the classical ones:
\begin{equation}
\label{const}  Q(0) = 0  \quad  \text{if $\;Q$ is continuous at} \;\;
0.
\end{equation}
\end{Remark}

Next we give the Fourier representation of the operators $L_{Dir}$
and their domains (see \cite[Prop.15]{DM16}). Notice that the matrix
of the operator $L_{Dir}$ (see (\ref{d.5}) below) does not depend on
the choice of constants discussed in the above Remark.
\begin{Proposition}
\label{prop002} In the above notations, if $ y \in H^1_{Dir}, $ then
we have $ y  = \sum_{k=1}^\infty y_k \sqrt{2} \sin kx \in D(L_{Dir})
$ and $ L y = h =  \sum_{k=1}^\infty h_k (y) \sqrt{2} \sin kx \in H^0
$ if and only if
\begin{equation}
\label{d.5} h_k (y) = k^2 y_k + \frac{1}{\sqrt{2}} \sum_{m=1}^\infty
\left ( \tilde{V} (|k-m|) - \tilde{V} (k+m) \right ) y_m
\end{equation}
and $   \sum |h_k (y)|^2 < \infty, $ i.e.,
\begin{equation}
\label{d.6} D(L_{Dir}) = \left \{ y \in H^1_{Dir} : \; (h_k
(y)_1^\infty \in \ell^2 (\mathbb{N}) \right \},
 \end{equation}
\begin{equation}
\label{d.7}
 L_{Dir} (y) = \sum_{k=1}^\infty h_k (y) \sqrt{2} \sin kx.
\end{equation}

\end{Proposition}
\bigskip

3. We study the equiconvergence  of  spectral decompositions of the
operators $L_{Per^\pm}, L_{Dir} $ and, respectively, $L^0_{Per^\pm},
\,L^0_{Dir} $ by using their Fourier representations with respect to
the corresponding o.n.b. (\ref{0.011})--(\ref{0.013}). In view of
Propositions~\ref{prop001} and \ref{prop002}, each of the operators
$L= L_{Per^\pm}, L_{Dir}$ has the form
\begin{equation}
\label{4.5} L = L^0 +V,
\end{equation}
where  the operators $L^0 $ and $V$ are defined, respectively, by
their action on the sequence of Fourier coefficients of  $y =
\sum_{\Gamma_{Per^\pm}} y_k \exp {ikx} \in H^1_{Per^\pm}$ or $y =
\sum_{\Gamma_{Dir}} y_k \sqrt{2} \sin kx \in H^1_{Dir}$ as follows:
\begin{equation}
\label{4.6} L^0: \; (y_k) \to (k^2 y_k), \quad k \in
\Gamma_{Per^\pm}, \; \Gamma_{Dir},
\end{equation}
\begin{equation}
\label{4.7} V:\; (y_m) \to (z_k), \quad z_k = \sum_{m} V(k-m) y_m,
\quad k, m \in \Gamma_{Per^\pm}
\end{equation}
for $bc = Per^\pm $ with $V(k)$  given by (\ref{0010}), and
\begin{equation}
\label{d.10} V:\; (y_m) \to (z_k), \quad z_k = \frac{1}{\sqrt{2}}
\sum_{m=1}^\infty \left ( \tilde{V} (|k-m|) - \tilde{V} (k+m) \right
) y_m, \quad k, m \in \mathbb{N}
\end{equation}
for $bc = Dir$ with  $\tilde{V} (k) $ given by (\ref{d.3}).

 These matrix representations could be used (see \cite{DM16,
DM21}) to justify the standard resolvent formula
\begin{equation}
\label{res} R_\lambda = R^0_\lambda + R^0_\lambda VR^0_\lambda +
\sum_{m=2}^\infty R^0_\lambda (VR^0_\lambda)^m, \quad \lambda \not
\in Sp (L_{bc}).
\end{equation}

Let $\Pi_N=\Pi_N (\omega, h), \; N \in \mathbb{N}, \; \omega,h>0 $
be the rectangle
\begin{equation}
\label{ec001} \Pi_N = \{z= x+iy: \; -\omega \leq x \leq N^2 +N, \;
\; |y|\leq h \},
\end{equation}
and let
\begin{equation}
\label{ec20}
 S_N    = \frac{1}{2 \pi i} \int_{\partial \Pi_N}
R_\lambda  d\lambda, \quad S^0_N    = \frac{1}{2 \pi i}
\int_{\partial \Pi_N} R^0_\lambda  d\lambda.
\end{equation}
The spectra of operators $L_{Per^\pm} $ are discrete; there are
numbers $\omega_0 = \omega_0 (v), \,h_0 = h_0 (v) $ and $N_0 = N_0
(v) $ such that for $ \omega \geq \omega _0, \, h\geq h_0  $ and
$N\geq N_0 $ the rectangle (\ref{ec001}) contains all periodic,
antiperiodic or Dirichlet eigenvalues which real part does not exceed
$N^2 +N$ (see \cite{DM16,DM21}).

By (\ref{res}) we have
\begin{equation}
\label{ec2a}
 S_N -S_N^0    = \frac{1}{2 \pi i} \int_{\partial \Pi_N}
(R_\lambda -R^0_\lambda) d\lambda = T_N  + B_N,
\end{equation}
where
\begin{equation}
\label{ec21} T_N = \frac{1}{2 \pi i} \int_{\partial \Pi_N}
R^0_\lambda VR^0_\lambda d\lambda
\end{equation}
\begin{equation}
\label{ec22} B_N  = \frac{1}{2 \pi i} \int_{\partial \Pi_N}
\sum_{m=2}^\infty R^0_\lambda (VR^0_\lambda)^m d\lambda.
\end{equation}
The representation (\ref{ec2a}) is crucial for our approach to
equiconvergence in the case of singular potentials. The operator
$B_N$ gives the "easy" part $S_N-S_N^0;$ we estimate from above its
norm by integrating the norm of the integrand in (\ref{ec22}).
However, when estimating the norm of $T_N $  it is essential first to
integrate over $\partial \Pi_N $ using residuum techniques -- see
Sections 3--5 for details.
\bigskip

4. How do the deviation-operators
\begin{equation}
\label{17}  S_N - S_N^0 : X \to Y, \quad  N \to \infty,
\end{equation}
behave for different pairs of functional Banach spaces?\\
 How does this behavior depend on the potential $v$, or
  on parameters $p$ or $\alpha$ if
\begin{equation}
\label{18} v \in L^p, \; \; 1 \leq p \leq 2, \quad \text{and} \quad
  v \in H^{- \alpha}, \; \; 0  < \alpha \leq 1?
\end{equation}
If $Y =\mathcal{C}= C([0, \pi])$ or $L^\infty ([0, 1])$ in
(\ref{17}) we speak on {\em uniform equiconvergence}.

V. A. Marchenko \cite{Ma86} proved, in the case $bc = Dir,$ that
\begin{equation}
\label{19}  \|(S_N - S_N^0)f \|_\infty \to 0
\end{equation}
if $v \in L^1 ([0, \pi])$ and  $f \in L^2([0, \pi]).$  V. A.
Vinokurov and
  V. A. Sadovnichii \cite{ViSa01} showed (\ref{19})
in the case when $bc = Dir $,  $v$ is real-valued such that
\begin{equation}
\label{20} v = Q^\prime \quad \text{with} \; Q \; \text{being a
periodic function of bounded variation},
\end{equation}
and $f \in L^1$.

One of the main results of our paper is the following assertion
(see Theorems~\ref{Thm2.1} and \ref{Thm3.10})

{\em Suppose $\,v $ is complex-valued, and $bc = Dir, Per^+ $ or $
Per^-$. If $v \in L^1$ then
\begin{equation}
\label{21}  \|S_N - S_N^0 : L^1 \to \mathcal{C} \| \to 0.
\end{equation}

If $v$ satisfies (\ref{20}), then
\begin{equation}
\label{22} \|(S_N - S_N^0)f \|_{\infty} \to 0  \quad    \forall  f
\in L^1([0, \pi]).
\end{equation}}

   Notice that in this statement
\begin{itemize}
\item $bc$ is not only $Dir$ but $Per^+$ and
$Per^-$ as well;

\item  $v$ is complex-valued;

\item  if $v \in L^1$ the claim (\ref{21}) is made for norms of the
deviation-operators.

\end{itemize}

(The latter means an improvement of
Tamarkin's second theorem in the case of Hill operators as well.)

  For the family of $L^p$-spaces we extend (\ref{19})
  to claim (see Theorem~\ref{Thm2.12}
  and  Corollary~\ref{Cor2.13} in Sect. 2) the following:

  {\em  Let $v \in L^p, \; 1 \leq p \leq 2,\;$  $1 \leq a \leq  b
  \leq \infty$ and
\begin{equation}
\label{23}  1/p + 1/a - 1/b <2.
\end{equation}
  Then
\begin{equation}
\label{24}  \|S_N - S_N^0: L^a \to L^b \| \leq C \|v\|_p \,
N^{-\gamma},
\end{equation}
  where}
\begin{equation}
\label{25}\gamma = (1 - 1/p) + (1 - (1/a - 1/b)).
\end{equation}

  In Marchenko's case (\ref{19})
  $$p=1, a=2, b=\infty \quad \text{so} \; \; \gamma = 0 + (1- 1/2) =
  1/2;$$
therefore (\ref{24}) and (\ref{25}) imply that
\begin{equation}
\label{26} \|S_N - S_N^0 \,: L^2 \to \mathcal{C} \| \leq C \|v\|_1
\, N^{-1/2}
\end{equation}
For $bc = Dir$  I. V. Sadovnichaya \cite{Sa09, 1-Sa10} considered
the problem of {\em uniform equiconvergence} for Hill operators,
respectively, with singular potentials $v \in H^{- \alpha}, \; 1/2 <
\alpha < 1 $ and $ v \in H^{-1}; $ see related papers
\cite{1-Sa08,2-Sa08,2-Sa10,Sa11} also.

   We extend analysis to $bc = Per^+$ and $Per^-$ and prove
   uniform equiconvergence as
\begin{equation}
\label{27} \|S_N - S_N^0: L^1 \to \mathcal{C} \|  \to 0 \quad
\text{if} \;\; v \in H^{- 1/2},
\end{equation}
and moreover, for $ v \in H^{- \alpha}, \; \frac{1}{2} \leq \alpha <1 $
we have
\begin{equation}
\label{28} \|S_N - S_N^0 : L^a \to \mathcal{C} \|  \to 0, \quad 1/a=
3/2 -\alpha.
\end{equation}
(See more precise and
complete claims in Sections 3 and 4.)
  The cases $ v \in H^{-1}, \; bc  = Per^\pm, $ remain unsolved although
for $bc = Dir $ it has been successfully done by
 I. V. Sadovnichaya \cite{1-Sa10}. \bigskip

{\em Remark.} In our main statements on {\em uniform} equiconvergence
(Theorems~\ref{Thm2.1}, \ref{Thm3.10}, \ref{thmec}) the proofs give
stronger claims on {\em absolute} convergence of Fourier coefficient
sequences $(f_k)$, so the $L^\infty$-norms in the image-spaces could
be changed to the Wiener norms $\|f\|_W = \sum |f_k|$. The inequality
(\ref{0.051}) guarantees that the Wiener norm is stronger than the
$L^\infty$-norm.
 \bigskip

5. Multidimensional analogs of the above questions are more
complicated because the structure of the spectrum and eigenfunctions
of the free operator, say, in the case of $L = - \Delta + v(x),\,$ $
x \in G \subset \mathbb{R}^m,$ $ G$ being a good bounded domain in
$\mathbb{R}^m,$ by itself is formidable problem -- see for example
\cite{A63}. It does not give any ready answers to be used in
analysis of equiconvergence. Still let us mention \cite{MS91} where
one can find an example of multidimensional equiconvergence in the
case of polyharmonic operators $(- \Delta)^a$ under strong
assumptions on the dimension $m$ and the order $2a$. Moreover, in
the case of 1D Dirac operators $L$ (see \cite{DM15, Mit04}), when
basic spectral properties of the free operator $L^0$ subject to
periodic, antiperiodic or Dirichlet $bc$ are well known, a series of
statements on equiconvergence of spectral decompositions has been
proven in \cite{Mit04}. In \cite{DM23,DM24}, we considered the Dirac
operator $L$ subject to arbitrary regular $bc.$ We  constructed
canonical Riesz bases of root functions of $L^0,$  used these bases
to develop Fourier analysis of $L,$ proved existence of Riesz type
spectral decompositions of $L$ and established for potentials $v \in
H^\alpha, \, \alpha >0 $ uniform equiconvergence of the spectral
decompositions of Dirac operators $L$ and $L^0,$ subject to
arbitrary regular $bc.$

The general approach and framework in this paper are similar to
those in \cite{Mit04} (in the case $v \in L^p, \; 1 < p \leq 2 $)
and \cite{DM23,DM24} (in the case $v \in H^{-1}$). \bigskip

   {\em  Acknowledgements}. We thank Professor Arkadi Minkin
   for the very useful discussions
of many aspects of the equiconvergence theory.

P. Djakov acknowledges the hospitality  of Department of Mathematics
and the support of Mathematical Research Institute of The Ohio State
University, July - August 2011.

B. Mityagin acknowledges the support of the Scientific and
Technological Research Council of Turkey and the hospitality of
Sabanci University, April--June, 2011.

\section{The case of potentials $v \in L^p, \, p\in(1,2]. $}

 1. First we consider the case of potentials $v \in L^p ([0,\pi]) \, p\in(1,2], $
 which illustrates all  crucial steps in our scheme
 but technically is more simple.  We normalize the Lebesgue measure
 so that the interval $[0,\pi]$ has measure one, and set
 $$
\|f\|_p = \|f| L^p \|= \left ( \frac{1}{\pi} \int_0^\pi |f(x)|^p
\right )^{1/p}.
 $$

If $F$ and $G$ are two functions then we write $F \lesssim G  $  for
$x \in D $ (or simply $F \lesssim G  $ when  $D$ is known by the
context) if there is a constant $C>0 $ such that
$$                      F(x) \leq C\cdot G(x) \quad  \forall x \in D. $$
We write $F \sim G $ if we have simultaneously $F \lesssim G  $ and
$G \lesssim F.  $

  \begin{Theorem}
\label{Thm2.1}
 Let a potential $v$ be an $L^p$-function,
 $ 1 < p \leq 2, \; 1/p + 1/q =
 1.$ Then, for $bc= Per^\pm, \, Dir,$
\begin{equation}
\label{2.1} \|S_N - S_N^0 \,: L^1 \to \mathcal{C}\| \lesssim N^{-
1/q}, \quad N \geq N_* (\|v\|_p).
\end{equation}
\end{Theorem}

\begin{proof} By (\ref{ec001})--(\ref{ec22}),
\begin{equation}
\label{2.2} S_N - S_N^0 = \frac{1}{2 \pi i} \int_{\partial \Pi_N}
(R(z) - R^0(z)) dz,
\end{equation}
where
\begin{equation}
\label{2.4} R(z) - R^0 (z) =  \sum_{m=2}^\infty U(m),
\end{equation}
and
\begin{equation}
\label{2.5} U(1) = R^0, \quad U(k+1) = U(k) V R^0, \quad k \geq 1
\end{equation}
with understanding that (\ref{2.2})--(\ref{2.5}) hold if all the
operators are well-defined and the series and integrals do converge.
We justify the latter by using inequalities proven in Section 6,
Appendix.

   The following diagrams  help:
\begin{equation}
\label{2.6} \ell^r \stackrel{\mathcal{F}^{-1}}{\longrightarrow}
L^\rho \stackrel{V}{\longrightarrow} L^r
\stackrel{\mathcal{F}}{\longrightarrow} \ell^\rho
\stackrel{\tilde{R}^0}{\longrightarrow} \ell^r, \quad D=\tilde{R}^0
\mathcal{F} V \mathcal{F}^{-1},
\end{equation}
with $r$ and $\rho $ chosen so that
\begin{equation}
\label{2.7} (a) \; \; \frac{1}{r} + \frac{1}{\rho} = 1 \quad
\text{and} \quad (b) \;\; \frac{1}{\rho} + \frac{1}{p} = \frac{1}{r}
\;\;  \Rightarrow (c) \; \; 1 > \frac{1}{r} = \frac{1}{2} \left (1 +
\frac{1}{p} \right ) > \frac{1}{2},
\end{equation}
and, for $m \geq 2,$
\begin{equation}
\label{diag.6} L^1 \stackrel{\mathcal{F}}{\longrightarrow}
\ell^\infty \stackrel{\tilde{R}^0}{\longrightarrow} \ell^r
\stackrel{D^{m-2}}{\longrightarrow} \ell^r
\stackrel{\mathcal{F}^{-1}}{\longrightarrow} L^\rho
\stackrel{V}{\longrightarrow} L^r
\stackrel{\mathcal{F}}{\longrightarrow} \ell^\rho
\stackrel{\tilde{R}^0}{\longrightarrow} \ell^1
\stackrel{\mathcal{F}^{-1}}{\longrightarrow} \mathcal{C},
\end{equation}
where
\begin{equation}
\label{2.8}   \mathcal{F} : f \to (\langle f, u_k \rangle)_{k \in
\Gamma_{bc}}
\end{equation}
puts in correspondence  to $f$ its sequence of Fourier coefficients
with respect to the canonical o.n.b. (\ref{0.011})--(\ref{0.013}),
  $$ \tilde{R}^0 : (t_k) \to {\frac{t_k}{z - k^2} : k \in
\Gamma_{bc} }$$ is a multiplier-operator in sequence spaces, and
$$  \mathcal{F}^{-1} : (t_k) \to \sum_{k \in \Gamma} t_k u_k(x) $$ is
the restoration of a function from the sequence of its Fourier
coefficients.

These are algebraic definitions but (\ref{2.7}a) and (\ref{0.051})
guarantee that
\begin{equation}
\label{F} \|\mathcal{F}: L^r \to \ell^\rho \| \leq \sqrt{2}, \quad
\|\mathcal{F}^{-1}: \ell^r \to L^\rho \| \leq \sqrt{2}
\end{equation}
(see Hausdorff-Young Theorem, \cite[Theorem XII.2.3]{Zy90}), and
(\ref{2.7}b) -- with H\"older Inequality -- shows that
\begin{equation}
\label{2.9}  \|V : L^\rho \to L^r \| = \|v\|_p.
\end{equation}
Analogously, in the case of multiplier-operators $ M: e_k \to m_k e_k
$ in sequence spaces we have, for $1 \leq  a < b \leq \infty, $
\begin{equation}
\label{mult} \|M: \ell^b \to \ell^a \| = \|(m_k)| \ell^c\| \quad
\text{with} \;\; 1/c+1/b = 1/a.
\end{equation}
Now Diagram (\ref{2.6}) shows that
\begin{equation}
\label{2.10} \|D: \ell^r \to \ell^r \|= \|\tilde{R}^0 \mathcal{F} V
\mathcal{F}^{-1}
  : \ell^r \to  \ell^r\| \leq 2 \|\tilde{R}^0|\ell^p\| \cdot \|v\|_p.
\end{equation}
Diagram (\ref{diag.6}) gives a factorization of the operator $U(m),
\, m \geq 2, $ so we obtain
\begin{equation}
\label{2.11} \|U(m) : L^1 \to \mathcal{C}\| \leq 4
\|\tilde{R}^0|\ell^r\| \cdot \|D\|^{m-2} \cdot \|\tilde{R}^0|\ell^r\|
\cdot \|v\|_p.
\end{equation}

Next we will use (\ref{2.10}), (\ref{2.11}) and inequalities from
Appendix to get estimates for the norms of operators
(\ref{2.2})--(\ref{2.4}).  The  horizontal sides and left vertical
side of $\partial \Pi_N $ could be sent to infinity (see Appendix,
Lemmas~\ref{lem405} and  \ref{lem407}), so
\begin{equation}
\label{2.12}  S_N - S_N^0 = \frac{1}{2 \pi i} \int_{\Lambda_N}
\sum_{m=2}^\infty U(m) dy
\end{equation}
with
\begin{equation}
\label{Lambda} \Lambda_N =\{z = N^2 + N + iy, \; y \in \mathbb{R} \}
\end{equation}
if we succeed to get good norm estimates on the line $\Lambda_N.$

 Notice that in (\ref{2.10}), for $z \in \Lambda_N, $
\begin{equation}
\label{2.13}
 \|\tilde{R}^0| \ell^p \| = A(z; p)=
 \left (\sum_{k\in \Gamma_{bc}} \frac{1}{|z-k^2|^p}\right )^{1/p},
 \quad  z=N^2 + N + iy,
\end{equation}
so we obtain, by Inequality (\ref{a56}),
\begin{equation}
\label{2.14}
  \|\tilde{R}^0| \ell^p\| \leq C(p) N^{-1}.
\end{equation}
Now (\ref{2.10}) and (\ref{2.14}) imply that there is $N_*=N_*
(\|v\|_p)$ such that for $N\geq N_*$ we have
\begin{equation}
\label{2.15} \|D\| \leq 2 C(p) N^{-1} \|v\|_p <1/4, \quad z=N^2 + N +
iy.
\end{equation}
Thus,  for $z=N^2 + N + iy$ with $N\geq N_*,$ it follows from
(\ref{2.11}) that
$$
\|U(m) : L^1 \to \mathcal{C}\| \leq  4^{3-m} \|\tilde{R}^0|\ell^r\|^2
\|v\|_p= 4^{3-m} A^2 (z;r)\|v\|_p,
$$
so by (\ref{2.12})
\begin{equation}
\label{2.17} \|S_N - S_N^0 : L^1 \to \mathcal{C}\| \leq 8\|v\|_p
\int_{\Lambda_N} A^2 (z;r) dy.
\end{equation}
Corollary~\ref{cor17A}, the third line in (\ref{108}), asserts that
\begin{equation}
\label{2.18}
  \int_{\Lambda} A^2 (z;r)  dy \leq C(r)  N^{-2(1-1/r)}.
\end{equation}
By (\ref{2.7}c), we have
\begin{equation}
\label{2.20} 2(1-1/r) = 2-(1+1/p)= 1/q.
\end{equation}
Therefore, (\ref{2.17}), (\ref{2.18}) and (\ref{2.20}) imply
(\ref{2.1}), which completes the proof.
\end{proof}
\bigskip

 2. Next we estimate the norms $\|S_N - S_N^0: L^a \to L^b\|, $ where
$L^a $ and $L^b$ are intermediate spaces such that
\begin{equation}
\label{2.25}
  L^1 \supset L^a \supset L^b   \supset L^\infty, \quad
  1 \leq a  \leq  2 \leq b \leq \infty \;\;  \text{with} \;\; 1/a - 1/b <
  1.
\end{equation}
Now we consider the following factoring of $U(m):$
\begin{equation}
\label{diag.26} L^a \stackrel{\mathcal{F}}{\longrightarrow}
\ell^\alpha \stackrel{\tilde{R}^0}{\longrightarrow} \ell^\tau
\stackrel{D^{m-2}}{\longrightarrow} \ell^\tau
\stackrel{\mathcal{F}^{-1}}{\longrightarrow} L^t
\stackrel{V}{\longrightarrow} L^s
\stackrel{\mathcal{F}}{\longrightarrow} \ell^\sigma
\stackrel{\tilde{R}^0}{\longrightarrow} \ell^\beta
\stackrel{\mathcal{F}^{-1}}{\longrightarrow} L^\beta,
\end{equation}
where the operator $D$ is defined from the diagram
\begin{equation}
\label{2.26} \ell^\tau \stackrel{\mathcal{F}^{-1}}{\longrightarrow}
L^t \stackrel{V}{\longrightarrow} L^s
\stackrel{\mathcal{F}}{\longrightarrow} \ell^\sigma
\stackrel{\tilde{R}^0}{\longrightarrow} \ell^\tau, \quad D=
\tilde{R}^0 \mathcal{F} V \mathcal{F}^{-1}.
\end{equation}
The arrow-operators in the above diagrams act as bounded operators
between the corresponding Banach spaces (of functions or sequences)
if the following seven conditions hold:\\

 (C1)  $\quad 1/a + 1/\alpha = 1, \; a \leq 2$ (Hausdorff-Young);\\

 (C2)  $\quad 1/\alpha + 1/r = 1/\tau,$ (H\"older inequality)
  with $r$ chosen
       to measure the norm of the multiplier-operator $\tilde{R}^0$;\\

 (C3) $\quad 1/\tau + 1/t = 1, \; \tau \leq 2$ (Hausdorff-Young);\\

 (C4) $\quad 1/t + 1/p = 1/s,$ (H\"older inequality) with
 $v \in L^p;$\\

 (C5)  $\quad  1/s + 1/\sigma = 1, \; s \leq 2$ (Hausdorff-Young);\\

 (C6) $\quad 1/\sigma + 1/r = 1/\beta,$  (H\"older inequality);\\

 (C7)  $\quad 1/\beta + 1/b = 1, \; \beta \leq 2$
 (Hausdorff-Young).\\

One can easily see that (C1)--(C7) imply (together with (\ref{2.25}))
\begin{equation}
\label{2.28}
 1/r= \frac{1}{2}(1/a -1/b + 1/p)< 1, \quad p \in [1,2].
\end{equation}
Moreover, if $r$ is given by (\ref{2.28}), then the parameters $\tau,
t, s, \sigma $ are uniquely determined by (C2)-(C5), and we have
$\tau \leq 2,$ $s\leq 2.$

As in the proof of Theorem \ref{Thm2.1} we want to use
(\ref{2.2})--(\ref{2.4}) and prove that the series on the right side
below converges and
\begin{equation}
\label{2.29} S_N - S_N^0 = \frac{1}{2 \pi } \int_{\Lambda_N}
\sum_{m=2}^\infty U(m) dy.
\end{equation}
Since $ \|\tilde{R}^0  : \ell^\sigma \to \ell^\tau\| =
 \|\tilde{R}^0 | \ell^p\|,$ from (\ref{2.26}) and (\ref{2.13}) it follows that
\begin{equation}
\label{2.30}
 \|D  : \ell^\tau \to \ell^\tau\| \leq
 \|\tilde{R}^0 | \ell^p\| \cdot \|v\|_p = A(z;p) \cdot \|v\|_p, \quad
 z \in \Lambda_N.
\end{equation}
Lemma~\ref{lemA}, (\ref{a55}) and (\ref{a56}) in Appendix show that
even in the worst case, for $ p \geq 1,$
\begin{equation}
\label{2.34} \|\tilde{R}^0 | \ell^p\|=A(z;p) \lesssim  \frac{\log
N}{N}, \quad z \in \Lambda_N.
\end{equation}
Therefore, in view of (\ref{2.30}) and (\ref{2.34}), there is $N_*$
such that
\begin{equation}
\label{2.35}
 \|D  : \ell^\tau \to \ell^\tau\| <
1/2, \quad N \geq N_*, \; \;
 z \in \Lambda_N.
\end{equation}

We have chosen $r$ so that  the norms in (C2) and (C6) are equal:
\begin{equation}
\label{2.36} \|\tilde{R}^0 : \ell^\alpha \to \ell^\tau\| =
\|\tilde{R}^0 : \ell^\sigma \to \ell^\beta\| = \|\tilde{R}^0 |
\ell^r\|.
\end{equation}
Now (\ref{diag.26}) together with (\ref{2.28})--(\ref{2.36}) show
that
$$ \|U(m) : L^a \to L^b\| \leq (1/2)^{m-2} \|v\|_p
\|\tilde{R}^0| \ell^r\|^2  \quad \text{for} \;N \geq N_*, \;  z \in
\Lambda_N.
$$
Therefore, by (\ref{2.13}) and (\ref{2.29}),
\begin{equation}
\label{2.38} \|S_N - S_N^0: L^a \to L^b \| \leq \int_{\Lambda_N}
  \sum_{m=2}^\infty \|U(m)\| dy \leq 2 \|v\|_p \int_{\Lambda_N} A^2(z;
  r)dy.
\end{equation}

Corollary \ref{cor17A}, see Appendix, gives estimates for
$\int_\Lambda A^2 (z;r)dy. $ In view of (\ref{2.28}), it leads us to
the following.
\begin{Theorem}
\label{Thm2.12} Suppose $v \in L^p, \;  1 \leq p \leq 2 $ and $bc
=Per^\pm $ or $bc = Dir. $ If (\ref{2.25}) holds, then
 $1/r = \frac{1}{2} (1/p + 1/a - 1/b) < 1$
and for $N \geq N_*(\|v\|_p)$
\begin{equation}
\label{2.43} \|S_N - S_N^0: L^a \to L^b \| \lesssim \begin{cases}
  \frac{1}{N}  &  \text{if} \;\; r >2;\\
\frac{\log N}{N}  &  \text{if} \;\; r =2;\\
  N^{-\gamma}&  \text{if} \;\; r < 2,
\end{cases}
\end{equation}
with
\begin{equation}
\label{2.43b} \gamma = (1-1/p) + (1-(1/a -1/b)).
\end{equation}
\end{Theorem}

\begin{Corollary}
\label{Cor2.13} If $1 < a \leq 2, \; b = \infty,$ and $ 1 \leq p < 2
$ then $1/r > 1/2 $ and by (\ref{2.43}) and (\ref{2.43b})
\begin{equation}
\label{2.44}
  \|S_N - S_N^0: L^a \to \mathcal{C}\|
  \lesssim N^{- \delta}, \quad
  \delta = (1 - 1/p) + (1 - 1/a).
\end{equation}
\end{Corollary}
\bigskip

 \section{$ L^1$-potentials and weakly singular potentials}

 1. Now we consider the uniform equiconvergence
 for functions in $L^1 $ in the case of
of $L^1$-potentials and potentials $v $ which are derivatives of
functions of bounded variation.

 J. Tamarkin \cite{Ta17, Ta28} proved --
 even in the more general case of higher order
ordinary differential operators  subject to Birkhoff-regular boundary
conditions -- that
\begin{equation}
\label{3.1} \text{if} \;\; v, \, f \in L^1 \;\; \text{then} \quad
\|(S_N - S_N^0)f \|_{C[0, \pi]} \to 0  \quad \text{as} \;\; N \to
\infty.
\end{equation}

We will show, for $bc=Per^\pm $ and $bc= Dir, $ that not only
strong convergence holds but the norm convergence to zero of the
deviation operators $S_N - S_N^0$ takes place as well. \bigskip

2. Let $v \in L^1, \; bc = Per^\pm $ or $Dir.$  We set
\begin{equation}
\label{3.2} v^*(n) = \begin{cases} \sup \{V(k): \; |k| \geq n\} &
\text{if} \;\; bc=Per^\pm,   \\ \sup \{\tilde{V}(k): \; k \geq n\}
&  \text{if}   \; \; bc = Dir,
\end{cases}
\end{equation}
where $V(k), \, k \in 2\mathbb{Z}$  and $\tilde{V}(k), \, k \in
\mathbb{N} $ are, respectively, the Fourier coefficients of $v(x) $
about the systems $\{e^{ikx}, \,k \in 2\mathbb{Z}\}$ and $\{\sqrt{2}
\cos kx, \, k \in \mathbb{N}\}.$

\begin{Theorem}
\label{Thm3.2} If $v \in L^1$ then, for $0 < H \leq N,$
\begin{equation}
\label{3.3}  \|S_N - S_N^0 : L^1 \to \mathcal{C} \| \lesssim v^*(H)
+  \frac{H}{N}.
\end{equation}

In particular, if $H = N^\gamma, \; 0 < \gamma < 1,$ then we have
\begin{equation}
\label{3.4} \|S_N - S_N^0 : L^1 \to \mathcal{C} \| \lesssim
v^*(N^\gamma) +  N^{\gamma - 1}.
\end{equation}
\end{Theorem}

\begin{proof}
By (\ref{ec2a}), $S_N - S_N^0= T_N + B_N, $ where $T_N$ and $B_N$
are given, respectively, by (\ref{ec21}) and (\ref{ec22}).

In Formula (\ref{ec22}), the horizontal sides and left vertical side
of $\partial \Pi_N $ could be sent to infinity (see Appendix,
Lemmas~\ref{lem405} and \ref{lem407}), so it follows that
\begin{equation}
\label{3.012}  B_N = \frac{1}{2 \pi i} \int_{\Lambda_N}
\sum_{m=3}^\infty U(m) dy.
\end{equation}
As in Section 2, we analyze the series under the integral in
(\ref{3.012}) by using diagrams. We factor the operator $U(m), \;
m\geq 1, $ by the diagram
\begin{equation}
\label{3.5} L^1 \stackrel{\mathcal{F}}{\longrightarrow} \ell^\infty
\stackrel{\tilde{R}^0}{\longrightarrow} \ell^1
\stackrel{D^{m-1}}{\longrightarrow} \ell^1
\stackrel{\mathcal{F}^{-1}}{\longrightarrow} \mathcal{C}, \quad
U(m)=\mathcal{F}^{-1}  D^{m-1} \tilde{R}^0  \mathcal{F},
\end{equation}
where the operator $D: \ell^1 \to \ell^1$ is defined by
\begin{equation}
\label{3.6} D= \tilde{R}^0 \mathcal{F} V \mathcal{F}^{-1}, \quad
\ell^1 \stackrel{\mathcal{F}^{-1}}{\longrightarrow} L^\infty
\stackrel{V}{\longrightarrow}  L^1
\stackrel{\mathcal{F}}{\longrightarrow} \ell^\infty
\stackrel{\tilde{R}^0}{\longrightarrow}
 \ell^1.
\end{equation}
In view of (\ref{0.051}), it follows that
\begin{equation}
\label{3.07}
 \|\mathcal{F}: L^1 \to
\ell^\infty\| \leq \sqrt{2}, \quad \|\mathcal{F}^{-1}: \ell^1 \to
L^\infty\| \leq \sqrt{2}.
\end{equation}
Therefore, by (\ref{3.6}) and (\ref{2.13}) we obtain
\begin{equation}
\label{3.7}
 \|D\| \leq 2 \|\tilde{R}^0 | \ell^1\| \cdot \|v\|_1=2 A(z;1) \cdot
 \|v\|_1.
\end{equation}
By Lemma~\ref{lemA}, (\ref{a55}) in Appendix,  we have that
\begin{equation}
\label{3.8} \|\tilde{R}^0 | \ell^1\|=A(z;1) \lesssim  \frac{\log
N}{N}, \quad z \in \Lambda_N.
\end{equation}
In view of (\ref{3.7}) and (\ref{3.8}), there is $N_*$ such that
\begin{equation}
\label{3.9}
 \|D  : \ell^1 \to \ell^1 \| <
1/2, \quad N \geq N_*, \; \;
 z \in \Lambda_N.
\end{equation}
Now, by (\ref{3.5}), (\ref{3.07})--(\ref{3.9}) it follows for $m\geq
3, \, N \geq N_*, \, z \in \Lambda_N, $ that
$$
\|U(m)\| \leq 2 \|\tilde{R}^0 | \ell^1\| \cdot \|D\|^2 \cdot
\|D\|^{m-3} \leq 8 (1/2)^{m-3} \|v\|_1^2 A^3 (z;1),
$$
which yields, in view of (\ref{3.012}),
$$
\|B_N: L^1 \to \mathcal{C}\| \leq \int_{\Lambda_N} \sum_{m=3}^\infty
\|U(m) \|dy \leq 16 \|v\|_1^2 \int_{\Lambda_N} A^3 (z;1)dy.
$$
Hence, from Corollary \ref{cor17B}  it follows that
\begin{equation}
\label{3.17} \|B_N: L^1 \to \mathcal{C}\| \lesssim 1/N.
\end{equation}
\bigskip

3. Next we need to analyze the operator $T_N.$ As before, we may
explain that $ T_N = \frac{1}{2 \pi i} \int_{\Lambda_N} U(2) dy. $
However, $\int_{\Lambda_N}  A^2 (z;1) dy = \infty, $ see (\ref{111})
in Appendix, so -- contrary to the case in Section 2 -- we cannot
integrate the estimate $\|U(2)\| \leq C A^2  (z;1) $ over $\Lambda_N$
and get an estimate of $\|T_N\|.$

We go around this bump by {\em integrating first} in (\ref{ec21})
 and then analyzing the resulting operator
by using its matrix representation with respect to the basis of
eigenfunctions of the free operator $L_{bc}^0$. Let
\begin{equation}
\label{3.19} f = \sum f_k u_k  \in L^1
\end{equation}
so if $bc = Per^+$ or $Per^-$
\begin{equation}
\label{3.20} R^0 V R^0 f= \sum_{m \in \Gamma_{bc}} \frac{1}{z-m^2}
\left ( \sum_{k \in \Gamma_{bc}} \frac{V(m-k) f(k)}{z-k^2} \right )
u_m.
\end{equation}

Our goal is to get the norm estimates, and if our results depend only
on the norms it is sufficient to check estimates on dense subsets in
$L^1,$ both for $f$ and for $v$. Therefore, one may assume that all
sums are finite.

Notice that
\begin{equation}
\label{3.21}
\frac{1}{2 \pi i} \int_{\partial \Pi_N}
\frac{dz}{(z-m^2)(z-k^2)} =
\begin{cases}
0 & \text{if} \; \;  |k|, |m| \leq N \;\text{or} \; |k|, |m| > N, \\
\frac{1}{m^2-k^2}   &   \text{if}   \;\;  |m| \leq N, \, |k|> N,\\
\frac{-1}{m^2-k^2}   &   \text{if}   \;\;  |m| > N, \, |k|\leq N.\\
\end{cases}
\end{equation}
Therefore,
the following holds.

\begin{Lemma}
\label{lin} For $bc = Per^\pm, $ the operator $T_N$ from
(\ref{ec21}) has a matrix representation
\begin{equation}
\label{3.21a} T_N (m,k) = \begin{cases} -\frac{V(m-k)}{|m^2 - k^2| },
& (m,k) \in X(N),\\
\quad  0,   &     (m,k) \not \in X(N),
\end{cases}
\end{equation}
respectively, about the o.n.b. $\{u_m: \, m \in \Gamma_{Per^\pm}\} $
of eigenfunctions of the free operator $L^0_{Per^\pm}, $ where
 \begin{equation}
\label{3.23}
 X(N)=\{(m,k): \; m,k \in \Gamma_{Per^\pm}; \; |m| \leq N, \; |k|>N  \;\;
 \text{or} \;\;  |m| > N, \; |k|\leq N \}.
\end{equation}
\end{Lemma}

We will use this matrix representation many times in what follows.
Now, in view of (\ref{3.19}) and (\ref{0.051}), we have
 \begin{equation}
\label{3.22} \|Tf\|_\infty = \left \|  \sum_{(m,k) \in X}
\frac{-V(m-k)f(k)}{|m^2-k^2| } u_m (x) \right \|_\infty  \leq
\sqrt{2} \, \|f\|_1  \sum_{(m,k) \in X} \frac{|V(m-k)|}{|m^2-k^2|}.
\end{equation}
By Lemma \ref{lem105},  Appendix,
 \begin{equation}
\label{3.24}
\sum_{(m,k) \in X(N)} \frac{1}{|m^2-k^2|} \leq 8 \cdot \frac{\pi^2}{8}
<10
\end{equation}
for any $N,$ and by Lemma~\ref{lem106}, Appendix,
 \begin{equation}
\label{3.25}
\sum_{\begin{array}{c} (m,k) \in X(N)\\ |m-k|\leq H \end{array}}
\frac{1}{|m^2-k^2|} \leq
4\cdot \frac{H}{N}.
\end{equation}
Therefore,
 \begin{equation}
\label{3.26} \sum_{(m,k) \in X(N)} \frac{|V(m-k)|}{|m^2-k^2|}=
\sum_{\begin{array}{c} (m,k) \in X(N)\\ |m-k| > H \end{array}} +
\sum_{\begin{array}{c} (m,k) \in X(N)\\ |m-k|\leq H \end{array}}
\end{equation}
$$\leq
10v^*(H) + \|v\|_1 \cdot  4 \frac{H}{N}.
$$
This completes the proof of (\ref{3.3}) if $bc = Per^\pm.$
\bigskip

4. The Dirichlet $bc$ is done in the same way but some adjustments
should be mentioned. For singular potentials, the matrix
representation of the multiplication operator $V$ comes from the
formulas (\ref{d.1}), (\ref{d.3}) and (\ref{d.10}). Of course, in the
classical case where
$$v(x) \in L^1,  \quad V(0) = \frac{1}{\pi}
\int_0^\pi v(x) dx = 0,$$ we have
$$v(x) = Q^\prime (x) \quad \text{with} \;\; Q(x) = \int_0^x v(t) dt. $$
Now one can easily see that (\ref{d.10}) holds with $(\tilde{V}(k))$
being the cosine coefficients of $v(x),  $  i.e.,
$$
V u_k = \sum_{m=1}^\infty V_{mk} u_m, \quad u_k =\sqrt{2} \sin kx,
$$
where
  $$
V_{mk}  = \frac{1}{\sqrt{2}}  \left (\tilde{V}(|m-k|) -\tilde{V}(m+k)
 \right ), \quad \tilde{V} (k) = \frac{1}{\pi} \int_0^\pi v(x)
 \sqrt{2} \cos kx dx.
$$
If $f =\sum_{k=1}^\infty f(k) u_k $ it follows that
\begin{equation}
\label{3.20a} R^0 V R^0 f= \sum_{m \in \Gamma_{bc}} \frac{1}{z-m^2}
\left ( \sum_{k \in \Gamma_{bc}} \frac{V_{mk} f(k)}{z-k^2} \right )
u_m.
\end{equation}
By (\ref{3.21}), after integration we obtain the following matrix
representation of the operator $T_N.$

\begin{Lemma}
\label{linD} For $bc = Dir, $ the operator $T_N$ from (\ref{ec21})
has a matrix representation
\begin{equation}
\label{3.21d} T_N (m,k) = \begin{cases}
\frac{\tilde{V}(m+k)-\tilde{V}(|m-k|)}{\sqrt{2}|m^2 - k^2| },
& (m,k) \in X(N),\\
\quad  0,   &     (m,k) \not \in X(N),
\end{cases}
\end{equation}
about the o.n.b. $\{\sqrt{2} \sin kx, \, k \in \mathbb{N}\}, $
where
 \begin{equation}
\label{3.23d}
 X(N)=\{(m,k): \; m,k \in \mathbb{N}: \; k \leq N< m \;\; \text{or}
 \;\; m \leq N <k\}.
\end{equation}
\end{Lemma}
With Formula (\ref{3.21d}) a proper adjustment in inequalities
(\ref{3.22})--(\ref{3.26}) leads to the estimate (\ref{3.3}) in the
case $bc =Dir.$
\end{proof}
\bigskip

5. The case where the potential $v$ is a derivative of a
$BV$-function.

In the case of Dirichlet $bc$ and a real-valued potential $v = Q', $
where $Q$ is a $\pi$-periodic function of bounded variation on
$[0,\pi],$ i.e.,
 \begin{equation}
\label{3.28}  Var (Q, [0,\pi])=
 \sup \{ \sum_{i=1}^n |Q(x_i)-Q(x_{i-1}|: \; 0=x_0<x_1 <\cdots < x_n=\pi
\}<\infty,
\end{equation}
V. A. Vinokurov and V. A. Sadovnichii \cite{ViSa01} showed that
\begin{equation}
\label{3.29}
\| (S_N - S_N^0)f \|_\infty
 \to 0  \quad  \text{as} \;\; N \to
\infty  \quad \forall f \in L^1([0, \pi]).
\end{equation}
We consider $bc = Per^+$ or $Per^-$ as well and drop the requirement
for $v$ to be real-valued. The following is true.
\begin{Theorem}
\label{Thm3.10} Let $v = Q^\prime, $  where $Q$ is a complex-valued
function of bounded variation on $[0,\pi].$ Then, for $bc= Per^\pm $
and $bc = Dir,$  the equiconvergence (\ref{3.29}) holds.
 \end{Theorem}

\begin{proof}
Consider the diagram
\begin{equation}
\label{3.30}
 \mathcal{C}^* \stackrel{\mathcal{F}}{\longrightarrow}
\ell^\infty \stackrel{\tilde{R}^0}{\longrightarrow} \ell^1
\stackrel{J}{\longrightarrow}  \mathcal{C}
\stackrel{V}{\longrightarrow}  \mathcal{C}^*,
\end{equation}
where $\mathcal{C}^* $ is the space of continuous linear functionals
on $\mathcal{C} = C([0,\pi]). $  If $v \in \mathcal{C}^* $ and $f \in
\mathcal{C}, $ then the product $v\cdot f$ is an element of
$\mathcal{C}^* $  such that
$$
\langle v\cdot f, \varphi \rangle = \langle v, f \cdot\varphi \rangle
\quad \forall \varphi \in \mathcal{C}.
$$

As in the proof of Theorem~\ref{Thm3.2} we come to the conclusion
\begin{equation}
\label{3.31}
\|S_N -S_N^0: \, L^1 \to \mathcal{C} \| \leq
M \left (v^* (H) + \|v\|_{\mathcal{C}^*} \frac{H}{N} \right ).
\end{equation}
If $v \in L^1 $ then $ v^* (H)  \to 0 $ as $H \to 0. $ But for $v \in
\mathcal{C}^* $ we can only say that $v^* (H) \leq M_1 \,Var (Q,
[0,\pi]) $  (see \cite[Theorem II.4.12]{Zy90}).

With $H=N, $  we obtain from (\ref{3.31})
\begin{equation}
\label{3.32}
\|S_N -S_N^0: \, L^1 \to \mathcal{C} \| \leq
M_{bc},
\end{equation}
where $M_{bc}$ is a constant which does not
depend on $N.$
However,
\begin{equation}
\label{3.33} \|(S_N -S_N^0) \varphi\|_\infty  \to 0
\end{equation}
if $\varphi $ is smooth enough, say $\varphi \in C^2 ([0,\pi]), $
and the space $C^2 ([0,\pi]) $ is dense in $L^1 ([0,\pi]). $
This explains that (\ref{3.32}) leads to (\ref{3.29}).
\end{proof}
\bigskip

6. The inequalities from Subsection 3.3 could be adjusted to
analysis of "weakly singular" potentials $ v\in H^{-\alpha}, \; 0 <
\alpha \leq 1/2,$  and one may show for such potentials that $ \|S_N
-S_N^0: L^1 \to L^\infty \| \to 0$ as $N \to \infty.$ But we prefer
to analyze these potentials in the next section, together with
"strongly singular" potentials  $ v \in H^{-\alpha}, \; 1/2 < \alpha
< 1.$
\bigskip

\section{The case of potentials $v\in H^{-\alpha}, \; 0 <\alpha<1.$}

1. Here we study how the equiconvergence depends on the singularity
of $v $ (measured by the appropriate scale of Sobolev spaces).

Recall that if $\Omega= (\Omega (k))_{k \in \mathbb{Z}}  $ is a
sequence of positive numbers (weight sequence), one may consider the
weighted sequence space $$ \ell^2 (\Omega,2\mathbb{Z})= \left
\{x=(x_k): \; \sum_{k \in 2\mathbb{Z}} (|x_k|\Omega (k))^2 < \infty
\right \} $$ and the corresponding Sobolev space
\begin{equation}
\label{sob} H(\Omega) =\left  \{ f = \sum_{k \in 2\mathbb{Z}} f_k
e^{ikx} : \;\; (f_k) \in \ell^2 (\Omega) \right \}.
\end{equation}
In particular, consider the Sobolev weights
\begin{equation}
\label{sob1}
 \Omega_\alpha (k) = (1+k^2)^{\alpha/2}, \quad k
\in \mathbb{Z}, \quad \alpha \in \mathbb{R},
\end{equation}
and the logarithmic weights
\begin{equation}
\label{sob2}
 \omega_\beta (k) = (\log(e+|k|))^\beta, \quad k
\in \mathbb{Z}, \quad \beta \in \mathbb{R}.
\end{equation}
Let $H^\alpha $ and $h^\beta $ denote the corresponding Sobolev
spaces (\ref{sob}). Of course, $H^\alpha \subset h^\beta $ if $\alpha
>0 $  and    $h^\beta \subset H^\alpha $ if $\alpha <0   $ for any
$\beta.$

The following lemma will be useful.
\begin{Lemma}
\label{lemsob} Let $g \in C^1 ([0,\pi]).$

(a)  If $f\in H^\alpha,  \;  - 1/2 < \alpha < 1/2,   $ then $f \cdot
g \in H^\alpha. $

(b)  If $f\in h^\beta,  \;  - \infty < \beta < \infty,   $ then $f
\cdot g \in h^\beta . $

\end{Lemma}

Proof is given in \cite[Appendix]{DM24}.

 Now we consider potentials $v\in H^{-\alpha}, \; 0< \alpha
<1,$ i.e. $v \in H^{-1}_{per} $ and
\begin{equation}
\label{3.35} v= \sum_{k\in 2\mathbb{Z}} V(k) e^{ikx},\quad V(0)=0,
\quad \sum_k \frac{|V(k)|^2}{(1+k^2)^\alpha} < \infty.
\end{equation}
or equivalently (see (\ref{0010})), $v= Q^\prime $ and
\begin{equation}
\label{3.350} Q =\sum_{k\in 2\mathbb{Z}} q(k) e^{ikx}, \quad
V(k)=ikq(k), \quad \sum_k |q(k)|^2 (1+k^2)^{1-\alpha} < \infty.
\end{equation}
Notice that $v \in H^{-\alpha}$ if and only if $Q \in H^{1-\alpha}.
$

In the context of Dirichlet boundary conditions, we may consider the
spaces $\tilde{H}^{-\alpha},$ $ \; 0 < \alpha <1/2 $ or $1/2 <\alpha
< 1,$ of all potentials $v \in H^{-1}_{per} $ such that
\begin{equation}
\label{3.35d} v= \sum_{m=0}^\infty \tilde{V}(m) \sqrt{2}\cos mx,\quad
\tilde{V}(0)=0, \quad \sum_m \frac{|\tilde{V}(m)|^2}{(1+m^2)^\alpha}
< \infty,
\end{equation}
or equivalently (see Formula (\ref{d.3})), $v=Q^\prime $ for some
$Q$ such that
\begin{equation}
\label{3.350d} Q =\sum_{m\in \mathbb{N}}  \tilde{q}(m) \sqrt{2}\sin
mx, \quad \tilde{V}(m)=m \tilde{q}(m), \quad \sum_m |\tilde{q}(m)|^2
m^{2(1-\alpha)} < \infty.
\end{equation}
It turns out that for $0< \alpha < 1/2$ the choice of an additive
constant for $Q$ (see Remark~\ref{remDir} and (\ref{const})) is
essential. Indeed, then (\ref{3.350d}) and the Cauchy inequality
imply
$$
\sum_{m=1}^\infty |\tilde{q} (m)|  \leq \left ( \sum_m
|\tilde{q}(m)|^2 m^{2(1-\alpha)}  \right )^{1/2} \left ( \sum_m
 m^{2(\alpha-1)}  \right )^{1/2} < \infty, \quad 0 <\alpha <
 \frac{1}{2}.
$$
Therefore, if (\ref{3.350d}) holds with $\alpha \in (0, 1/2), $ then
the function $Q(x) $ is continuous, and $Q(0)=0. $

\begin{Proposition}
\label{propeq} If $\, 0 < \alpha < 1/2 $ or $\, 1/2 < \alpha < 1, $
then $\tilde{H}^{-\alpha} = H^{-\alpha}. $
\end{Proposition}

\begin{proof}
It is known that the discrete Hilbert transform
\begin{equation}
\label{HT} \mathcal{H}: \ell^2 (\mathbb{Z}) \to \ell^2 (\mathbb{Z}),
\quad  (\mathcal{H}x)_n = \sum_{k\neq n} \frac{x_k}{n-k}
\end{equation}
act continuously in the weighted spaces $\ell^2 (\mathbb{Z},
\Omega_\delta), \; \Omega_\delta (k)= (1+k^2)^{\delta/2} $ if
$|\delta | < 1/2 $ (\cite[Theorem 10]{HMW}, see also Lemma~32 in
\cite[Appendix]{DM24}). We use this fact several times in the proof.

First we show that  $\tilde{H}^{-\alpha} \subset H^{-\alpha}.$
Suppose $v \in \tilde{H}^{-\alpha}; $ then (\ref{3.350d}) holds, so
\begin{equation}
\label{3.351} (\tilde{q}(m))_{m\in \mathbb{N}} \in \ell^2
(m^{1-\alpha}, \mathbb{N}).
\end{equation}
Taking into account that
\begin{equation}
\label{FC} \int_0^\pi \sin mx \, e^{-i2k x} dx = \begin{cases}
0 & m= 2s>0, \; |k|\neq s;\\
 \frac{\pi}{2i} \, sgn (k) &  m=2|k|; \\
\frac{1}{2s-1-2k}+ \frac{1}{2s-1+2k} & m=2s-1,
\end{cases}
\end{equation}
we evaluate $ q(2k) = \frac{1}{\pi}\int_0^\pi \left (
\sum_{m=1}^\infty \tilde{q}(m) \sqrt{2} \sin mx \right ) e^{-i2k x}
dx: $
\begin{equation}
\label{eqq} q(2k) =\frac{-i}{\sqrt{2}} \tilde{q} (2|k|) \,sgn (k)
+\frac{\sqrt{2}}{\pi}\sum_{s=1}^\infty \tilde{q} (2s-1) \left (
\frac{1}{2s-1-2k} +\frac{1}{2s-1+2k}\right ).
\end{equation}

In the case $1/2 <\alpha <1,$ the latter sum can be regarded as a
discrete Hilbert transform of the sequence $\xi =(\xi_k)_{k \in
\mathbb{Z}},$ where
$$
\xi_k = 0 \quad \text{if $k$ is even} , \quad \xi_k =- sgn (k) \,
\tilde{q}(|k|) \quad \text{if $k$ is odd},
$$
that is, we have
$$
q(2k) =\frac{-i}{\sqrt{2}} \tilde{q} (2|k|) \,sgn (k)+
\frac{\sqrt{2}}{\pi}(\mathcal{H}\xi)_{2k}, \quad  (\mathcal{H} \xi)_n
=\sum_{k\neq n} \frac{\xi_k}{n-k}.
$$
Moreover, by (\ref{3.351}) we have $\xi \in \ell^2
(\Omega_{(1-\alpha)}, \mathbb{Z}) $ with $0<1-\alpha< 1/2,$  so it
follows that $\{(H\xi)_{2k}\}\in \ell^2 (\Omega_{(1-\alpha)},
2\mathbb{Z}).$ Therefore, (\ref{3.350}) holds, i.e., $v\in
H^{-\alpha}. $ Hence $\tilde{H}^{-\alpha} \subset H^{-\alpha} $ if
$1/2 <\alpha <1. $

In the case $0 < \alpha < 1/2, $ we multiply  (\ref{eqq}) by $(2k)i$
and obtain
\begin{equation}
\label{eqq1}
  V(2k) = \frac{1}{\sqrt{2}} \tilde{V}(2|k|) +
i\frac{\sqrt{2}}{\pi}\sum_{s=1}^\infty \tilde{V} (2s-1) \left (
\frac{1}{2s-1-2k} -\frac{1}{2s-1+2k}\right )
\end{equation}
because
$$ 2k  \left ( \frac{1}{2s-1-2k} + \frac{1}{2s-1+2k}\right )= (2s-1)  \left (
\frac{1}{2s-1-2k} -\frac{1}{2s-1+2k}\right ).
$$
The sum in (\ref{eqq1}) may be considered as a discrete Hilbert
transform of the sequence $u= u_k,$ where
$$
u_k = 0 \quad \text{if $k$ is even} , \quad u_k =- \tilde{V}(|k|)
\quad \text{if $k$ is odd},
$$
that is, we have
$$
V(2k) =\frac{1}{\sqrt{2}} \tilde{V} (2|k|) \,sgn (k)+
\frac{i\sqrt{2}}{\pi}(\mathcal{H} u)_{2k}.
$$
Since $(\tilde{V}(m)) \in \ell^2 (\Omega_{-\alpha}),\; 0 < \alpha <
1/2, $ we obtain
$$
u \in \ell^2 (\Omega_{-\alpha}) \;\; \Rightarrow \;\; \mathcal{H} u
\in \ell^2 (\Omega_{-\alpha})\;\; \Rightarrow \;\; (V(2k)) \in \ell^2
(\Omega_{-\alpha}) \;\; \Rightarrow \;\;  v \in H^{-\alpha}.
$$
This completes the proof of the inclusion $\tilde{H}^{-\alpha}
\subset H^{-\alpha}.$\bigskip

Next we show that $\tilde{H}^{-\alpha} \supset H^{-\alpha}.$
 Let $v \in H^{-\alpha}, \; 1/2 <\alpha <1. $
 Since $\tilde{q}(m)=\frac{1}{\pi} \int_0^\pi Q(x) \sqrt{2}\sin mx
\,dx, $ we obtain
$$
\tilde{q}(m) =\begin{cases} \frac{i}{\sqrt{2}} [(q(2k)- q(-2k)] &
\text{if} \;\; m =2k,\\ \frac{i}{\sqrt{2}} [q_1 (2k+2) - q_1 (-2k)]
&\text{if} \;\; m= 2k+1,
\end{cases}
$$
where $\{q_1 (s), \, s \in 2\mathbb{Z}\}$ are the Fourier
coefficients of the function $Q_1 (x) = e^{ix} \cdot Q(x).$ By
Lemma~\ref{lemsob}, we have
$$
Q\in H^{1-\alpha}  \Longleftrightarrow Q_1 \in H^{1-\alpha} \quad
\text{if} \;\;  1/2< \alpha < 1.
$$
Therefore, if $\frac{1}{2} < \alpha < 1 $ and $v \in H^\alpha,$ then
we obtain
$$
v \in H^{-\alpha} \Rightarrow  \text{(\ref{3.350})} \Rightarrow Q \in
H^{1-\alpha} \Rightarrow Q_1 \in H^{1-\alpha} \Rightarrow
\text{(\ref{3.350d})} \Rightarrow  \text{(\ref{3.35d})} \Rightarrow v
\in \tilde{H}^{-\alpha},
$$
i.e., $\tilde{H}^{-\alpha} \supset H^{-\alpha} $  if $\frac{1}{2} <
\alpha < 1. $

Next we show that $\tilde{H}^{-\alpha} \supset H^{-\alpha}  $ in the
case $0 < \alpha <\frac{1}{2}. $ Let $v \in H^{-\alpha}. $ Then, by
(\ref{3.350}) and the Cauchy inequality, $ \sum |q(2k)| < \infty, $
so $Q(x)= \sum q(2k) e^{i2kx} $ is continuous function and we have
\begin{equation}
\label{Dir0} Q(0) = \sum_{\mathbb{Z}} q(2k) =0  \;\; \Rightarrow \;\;
q(0) = - \sum_{k\neq 0} q(2k).
\end{equation}
We evaluate the coefficients $\tilde{q} (m), \; m \in \mathbb{N}:$
$$
\tilde{q} (m)= \frac{1}{\pi} \sqrt{2}\int_0^\pi Q(x) \sin mx \, dx =
\sum_{k\in \mathbb{Z}} q(2k) \frac{\sqrt{2}}{\pi} \int_0^\pi e^{i2kx}
\sin mx \, dx.
$$
In view of (\ref{FC}), we obtain
\begin{equation}
\label{d.91} \tilde{q} (m)= \frac{i}{\sqrt{2}}[q(m) - q(-m)] \quad
\text{for even} \;\; m,
\end{equation}
\begin{equation}
\label{d.92} \tilde{q} (m)= \frac{\sqrt{2}}{\pi} \sum_{k \in
\mathbb{Z}} q(2k) \left ( \frac{1}{m+2k}+\frac{1}{m-2k} \right )
\quad \text{for odd} \;\; m.
\end{equation}
By (\ref{Dir0}), (\ref{d.92}) implies
\begin{equation}
\label{d.93} \tilde{q} (m)= \frac{\sqrt{2}}{\pi} \sum_{k \neq 0}
q(2k) \left ( \frac{1}{m+2k}+\frac{1}{m-2k} - \frac{2}{m} \right )
\quad \text{for odd} \;\; m.
\end{equation}
Since $\tilde{V} (m)=m \,\tilde{q} (m), \; V(2k) = i(2k) q(2k)$ and
$$\frac{1}{m+2k}+\frac{1}{m-2k} - \frac{2}{m}=\frac{2k}{m}
\left ( \frac{1}{m-2k}-\frac{1}{m+2k} \right ),
$$
from (\ref{d.93}) it follows that
\begin{equation}
\label{d.95} \tilde{V} (m)= \frac{1}{i}\sum_{k \neq 0}
\frac{V(2k)}{m-2k} -\frac{1}{i}\sum_{k \neq 0} \frac{V(2k)}{m+2k}
=\frac{1}{i} \left [(\mathcal{H}w)_m - (\mathcal{H}w)_{-m} \right ]
\quad \text{for odd} \;\; m,
\end{equation}
where
$$
w_s =\begin{cases}   0   & \text{if} \;\;   s= 2k-1,\\ V(2k) &
\text{if} \;\; s=2k.
\end{cases}
$$
By  (\ref{3.35}), we know that $w \in \ell^2 (\Omega_{-\alpha},
\mathbb{Z}), $  so $\mathcal{H}w \in \ell^2 (\Omega_{-\alpha},
\mathbb{Z}) $ also. Therefore, by (\ref{d.91}) and (\ref{d.95}) we
conclude that $(\tilde{V}(m)) \in \ell^2 (\Omega_{-\alpha},
\mathbb{Z}),$  i.e., $v \in \tilde{H}^{-\alpha}.$ Hence,
$\tilde{H}^{-\alpha} \supset H^{-\alpha}  $ if $0 <\alpha < 1/2.$
This completes the proof.
\end{proof}

\begin{Remark}
\label{remH} The definition (\ref{3.35d}) of the classes
$\tilde{H}^{-\alpha}$ for $1/2 <\alpha < 1$ is correct (although if
we add a constant $C$ to $Q$ then for odd $m$ the coefficients
$\tilde{V}(m)$ will change by $2C$). But we cannot define a class
$\tilde{H}^{-1/2} $ by (\ref{3.35d}) with $\alpha = 1/2 $ because
such a definition will depend essentially on the choice of an
arbitrary additive constant.
\end{Remark}

\bigskip

2. Our main result in this section is the following theorem.

\begin{Theorem}
\label{thmec} Let $S_N, \, S_N^0 $ be the spectral projections
defined by (\ref{ec20}) for the Hill operators $L_{bc} (v) $ and
$L_{bc}^0 $ subject to the boundary conditions $bc = Per^\pm $ or $
Dir. $

(a) If  $v \in H^{-\alpha} $ with $ \alpha \in (0,1/2),  $ then
\begin{equation}
\label{ec1} \|S_N-S_N^0: L^1 \to L^\infty \| =o(N^{\alpha
-\frac{1}{2}}), \quad N\to \infty.
\end{equation}

(b) If $bc = Per^\pm $ and $v \in H^{-1/2} $ or $bc = Dir $ and
$v=Q^\prime $ with $Q= \sum_{m\in \mathbb{N}} \tilde{q} (m) \sin mx,
\;$ $\sum_{m\in \mathbb{N}} |\tilde{q}(m)|^2 \,m <\infty, $ then
\begin{equation}
\label{ec2} \|S_N - S_N^0: L^1 \to L^\infty \| \to 0 \quad \text{as}
\;\; N \to \infty.
\end{equation}

(c)  If  $v \in H^{-\alpha} $ with $ \alpha \in (1/2,1)  $ and
$a=\frac{2}{3-2\alpha} $ (i.e., $\frac{1}{a}=\frac{3}{2}-\alpha$),
then
\begin{equation}
\label{ec3} \|S_N - S_N^0: L^a \to L^\infty \| \to 0 \quad \text{as}
\;\; N \to \infty.
\end{equation}
\end{Theorem}

\begin{proof}
By (\ref{ec2a}), we have $S_N - S_N^0= T_N + B_N, $ where
 the operators $T_N $ and $B_N $ are defined by (\ref{ec21})
and (\ref{ec22}). We estimate appropriate norms of the operators
$T_N$ and $B_N $ in Propositions~\ref{propBN} and \ref{propAN} below.
The results obtained there are of independent interest -- they are
more general than the estimates leading to (\ref{ec1})--(\ref{ec3}).

First we prove (a). Let $v \in H^{-\alpha} $ with $\alpha \in (0,
1/2). $ Then  Proposition~\ref{propeq} and
(\ref{3.35})--(\ref{3.350d}) imply that
\begin{equation}
\label{e55} q\in \ell^2 (\Omega), \quad \tilde{q} \in \ell^2
(\Omega), \quad \Omega (k) = (1+k^2)^{(1-\alpha)/2}.
\end{equation}
Therefore, by (b) in Proposition~\ref{propAN} (with $a=1$ and $
\delta= 1-\alpha$ in (\ref{e25})) we have
$$
\|T_N: L^1 \to L^\infty\| \lesssim \frac{H}{N^{1/2}}
+\mathcal{E}^\Omega_{H_N}(q) N^{\alpha-\frac{1}{2}}.
$$
So, choosing $H_N= N^{\alpha}/\log N $ we obtain
$$
\|T_N: L^1 \to L^\infty\| = o \left (N^{\alpha-\frac{1}{2}} \right )
$$
because $\mathcal{E}^\Omega_{H_N}(q) \to 0$ as $N \to 0.$

On the other hand, by (c) in Proposition~\ref{propBN} (see
(\ref{ecSob}) with $\delta= 1-\alpha $) we have
$$
\|B_N: L^1 \to L^\infty\| \lesssim  N^{\alpha-1} (\log N)^2 = o
\left (N^{\alpha-\frac{1}{2}} \right ).
$$
Hence (\ref{ec1}) holds.

By the assumption of (b), it follows  that (\ref{e55}) holds with
$\alpha = 1/2. $  Indeed, if $v \in H^{-1/2}, $ this follows from
(\ref{3.35}) and (\ref{3.350}), and it is assumed that $\tilde{q}
\in \ell^2 (\Omega)$ with $ \Omega (k) = (1+k^2)^{1/4}.$ Now the
same argument as in the proof of (a) shows that (\ref{ec2}) holds.

Finally, we prove (c). Let $v \in H^{-\alpha} $ with $\alpha \in (
1/2, 1). $ As in the proof of (a),  Proposition~\ref{propeq} and
(\ref{3.35})--(\ref{3.350d}) imply that (\ref{e55}) holds. Therefore,
by (b) in Proposition~\ref{propAN}  with $a=\frac{2}{3-2\alpha}, $ $
\delta= 1-\alpha $ and $H=N^{a/4}$ in (\ref{e25}), we obtain
$$
\|T_N: L^a \to L^\infty\| \lesssim \frac{1}{N^{1/4}} +
\mathcal{E}^\Omega_{N^{a/4}}(q) \to 0 \quad \text{as} \;\; N \to
\infty.
$$
On the other hand, in view (c) in Proposition~\ref{propBN} (see
(\ref{ecSob}) with $\delta= 1-\alpha $) we have
$$
\|B_N: L^a \to L^\infty\| \leq \|B_N: L^1 \to L^\infty\| \to 0 \quad
\text{as} \;\; N \to \infty.
$$
Hence (\ref{ec3}) holds, which completes the proof.
\end{proof}
\bigskip

3. Our proofs are based on the Fourier analysis approach to the
theory of Hill operators with singular potentials developed in
\cite{DM16}. Below we recall some basic formulas related to this
approach.

 In general, there are no good estimates for the norms of
 $      R^0_\lambda V  $ and
$ V R^0_\lambda $ in the case of singular potentials. Therefore, now
we write the standard perturbation type formula for the resolvent
$R_\lambda $ in the form
\begin{equation}
\label{4.23} R_\lambda = R^0_\lambda + R^0_\lambda V R^0_\lambda +
 \cdots = K^2_\lambda +
\sum_{m=1}^\infty K_\lambda(K_\lambda V K_\lambda)^m K_\lambda,
\end{equation}
where
\begin{equation}
\label{4.24} (K_\lambda)^2 = R^0_\lambda.
\end{equation}
We define an operator $K= K_\lambda $ with the property (\ref{4.24})
by its matrix representation
\begin{equation}
\label{4.25} K_{jm} = \frac{1}{(\lambda - j^2)^{1/2}}
\delta_{jm},\qquad j,m \in \Gamma_{bc},
\end{equation}
where $$z^{1/2} = \sqrt{r} e^{i\varphi/2} \quad \mbox{if} \quad z=
re^{i\varphi}, \;\; 0\leq \varphi < 2\pi. $$

Then $R_\lambda $ is well--defined if
\begin{equation}
\label{4.26} \|K_\lambda V K_\lambda:  \; \ell^2 (\Gamma_{bc})  \to
\ell^2 (\Gamma_{bc})\| < 1.
\end{equation}

In view of (\ref{4.7}) and (\ref{4.25}), the matrix representation
of $KVK$ for periodic or anti--periodic boundary conditions $bc =
Per^\pm$ is
\begin{equation}
\label{4.29} (KVK)_{jm} = \frac{V(j-m)}{(\lambda -
j^2)^{1/2}(\lambda - m^2)^{1/2}} =i\, \frac{(j-m)q(j-m)}{(\lambda -
j^2)^{1/2}(\lambda - m^2)^{1/2}},
\end{equation}
where $j,m \in 2\mathbb{Z} $ for $bc = Per^+,$ and $j,m \in 1+
2\mathbb{Z} $ for $bc=Per^-.$ Therefore, we have for its
Hilbert--Schmidt norm (which dominates its $\ell^2 $-norm)
\begin{equation}
\label{4.30} \|KVK\|_{HS}^2 = \sum_{j,m \in \Gamma_{Per^\pm}}
\frac{(j-m)^2 |q(j-m)|^2} {|\lambda - j^2| |\lambda - m^2|}.
\end{equation}

In the case $bc =Dir,$ we obtain by (\ref{d.10}) and (\ref{4.25})
that
\begin{equation}
\label{4.29a} (KVK)_{jm}= \frac{1}{\sqrt{2}}
\frac{|j-m|\,\tilde{q}(|j-m|)-(j+m) \,\tilde{q}(j+m)}{(\lambda -
j^2)^{1/2}(\lambda - m^2)^{1/2}}, \quad j,m\in \mathbb{N}.
\end{equation}
Thus,
\begin{equation}
\label{4.30a} \|KVK\|_{HS}^2 \leq \sum_{j,m \in \mathbb{N}}
\frac{(j-m)^2 |\tilde{q}(j-m)|^2+ (j+m)^2 |\tilde{q}(j+m)|^2}
{|\lambda - j^2| |\lambda - m^2|}.
\end{equation}

In view of (\ref{4.30}) and (\ref{4.30a}), we can estimate from
above the Hilbert-Schmidt norm $\|KVK\|_{HS}$ by one and the same
formula in all three cases $bc= Per^+, \, Per^-, \, Dir.$ Indeed, if
we set $q(k) = 0 $ for $k \in 2\mathbb{Z}+1 $ if $bc= Per^\pm,$ and
$q(k) =\tilde{q} (|k|) $ if $bc= Dir,$  then $q\in \ell^2
(\mathbb{Z})$ and we have
\begin{equation}
\label{Z} \|KVK\|^2_{HS} \leq \sum_{j,m \in \mathbb{Z}} \frac{(j-m)^2
|q(j-m)|^2} {|\lambda - j^2| |\lambda - m^2|}, \quad bc=Per^\pm,\;
Dir.
\end{equation}
\bigskip

4. Next we estimate the Hilbert-Schmidt norm of the operator
$K_\lambda VK_\lambda $  for $\lambda = N^2 +N +i y, \; y \in
\mathbb{R}.$ For a sequence $q=(q(k))\in \ell^2, $ or $q=(q(k))\in
\ell^2 (\Omega) $ we set
\begin{equation}
\label{E} \mathcal{E}_M (q) = \left (\sum_{|k|\geq M} |q(k)|^2 \right
)^{1/2}, \quad \mathcal{E}^\Omega_M (q) = \left (\sum_{|k|\geq M}
|q(k)|^2 (\Omega (k))^2 \right )^{1/2}.
\end{equation}

\begin{Lemma}
\label{lemBE}  For $q= (q(k))_{k\in \mathbb{Z}} \in \ell^2
(\mathbb{Z})$ we set
\begin{equation} \label{4.33}  \psi_N (y) =
\sum_{j,m \in \mathbb{Z}} \frac{(j-m)^2 |q(j-m)|^2} {|\lambda - j^2|
|\lambda - m^2|}, \quad \lambda = N^2 +N + iy.
\end{equation}
Then
\begin{equation}
\label{4.34} \psi_N (y) \leq  N^2 \left (  \frac{\|q\|^2}{N}  + 16
(\mathcal{E}_{\sqrt{N}}(q))^2 \right ) b_N (y) + 16
(\mathcal{E}_{4N}(q))^2 a_N (y),
\end{equation}
where
\begin{equation}
\label{4.35} a_N (y) = \sum_{k\in \mathbb{Z}} \frac{1}{|\lambda -
k^2|}, \quad b_N (y) = \sum_{k\in \mathbb{Z}} \frac{1}{|\lambda -
k^2|^2}, \quad \lambda = N^2 +N + iy.
\end{equation}
Moreover, if $|y| \geq N^8,$  then we have
\begin{equation}
\label{4.36} \psi_N (y) \leq  N^2 \left (  \frac{\|q\|^2}{N}  + 16
(\mathcal{E}_{\sqrt{N}}(q))^2 \right ) b_N (y) + \left
(\frac{\|q\|^2}{\sqrt{|y|}} + 16(\mathcal{E}_{|y|^{1/4}}(q))^2
\right ) a_N (y).
\end{equation}
\end{Lemma}

\begin{proof}
In view of (\ref{4.33}),
\begin{equation}
\label{4.38} \psi_N (y) = \sum_{s \in \mathbb{Z}}\left (\sum_{m \in
\mathbb{Z}} \frac{s^2 |q(s)|^2} {|\lambda - (m+s)^2| |\lambda -
m^2|} \right ) = \sigma_1  +\sigma_2 +\sigma_3,
\end{equation}
where
\begin{equation}
\label{4.39}  \sigma_1= \sum_{|s| \leq \sqrt{N}}  \cdots,  \quad
\sigma_2 = \sum_{\sqrt{N} < |s| \leq 4N} \cdots, \quad \sigma_3 =
\sum_{|s| > 4N} \cdots.
\end{equation}
The Cauchy inequality implies that
\begin{equation}
\label{4.40}   \sum_{m\in \mathbb{Z}} \frac{1}{|\lambda -m^2 | |
\lambda - (m+s)^2|} \leq \sum_{m\in \mathbb{Z}} \frac{1}{|\lambda
-m^2|^2}.
\end{equation}
Therefore, in view of (\ref{4.35}), it follows that
\begin{equation}
\label{4.41}  \sigma_1  \leq  \left (\sum_{|s| \leq \sqrt{N}} s^2
|q(s)|^2 \right ) b_N (y) \leq N \|q\|^2 b_N (y)
\end{equation}
and
\begin{equation}
\label{4.42}  \sigma_2  \leq  \left ( \sum_{\sqrt{N} < |s| \leq 4N}
s^2 |q(s)|^2 \right ) b_N (y) \leq (4N)^2
(\mathcal{E}_{\sqrt{N}}(q))^2 b_N (y).
\end{equation}
Next we estimate $\sigma_3.$  In view of (\ref{4.38}) and
(\ref{4.39}),
$$
\sigma_3  \leq (\mathcal{E}_{4N}(q))^2 \cdot \sup_{|s|>4N} \sum_m
\frac{s^2}{|\lambda - (m+s)^2| |\lambda - m^2|}.
$$
If $|s|> 4N $  and $\lambda = N^2 +N + iy,$ then
\begin{equation}
\label{4.44} \frac{s^2}{|\lambda - (m+s)^2| |\lambda - m^2|} \leq
\frac{8}{|\lambda - m^2|}+\frac{8}{|\lambda - (m+s)^2|}.
\end{equation}
Indeed, if $ |m| \geq |s|/2,$ then  (since $ |s|/4 >N  $)
$$|\lambda - m^2| \geq m^2 - |Re \, \lambda | \geq  s^2/4 -
(N^2 +N)
> s^2/4 - 2(|s|/4)^2 \geq s^2/8, $$
so (\ref{4.44}) holds. If $ |m| < |s|/2,$ then $|m+s| \geq |s|/2,$
and as above it follows that $|\lambda -(m+s)^2| \geq s^2/8, $ so
(\ref{4.44}) holds also. Therefore,  $$ \sup_{|s|>4N} \sum_m
\frac{s^2}{|\lambda - (m+s)^2| |\lambda - m^2|} \leq \sum_m
\frac{16}{|\lambda - m^2|} = 16 \, a_N (y),
$$
so we obtain
\begin{equation}
\label{4.45} \sigma_3  \leq 16 (\mathcal{E}_{4N}(q))^2  a_N (y).
\end{equation}
Now, in view of (\ref{4.38}), the estimates (\ref{4.41}),
(\ref{4.42}) and (\ref{4.45}) imply (\ref{4.34}).

Next we prove (\ref{4.36}). To this end we estimate $\sigma_3=
\sigma_3 (y) $ for $|y| > N^8. $ Then
$$
\sigma_3 = \sum_{4N <|s| \leq |y|^{1/4}} \cdots + \sum_{ |s|>
|y|^{1/4}} \cdots = \sigma_{3,1} + \sigma_{3,2}.
$$
If $|s|<|y|^{1/4},$ then  $$  \frac{s^2}{|\lambda - (m+s)^2|} \leq
\frac{|y|^{1/2}}{|Im \, \lambda|} = \frac{1}{|y|^{1/2}}, $$ so
$$
\sigma_{3,1} \leq \frac{1}{|y|^{1/2}} \left (\sum_{4N <|s| \leq
|y|^{1/4}} |q(s)|^2 \right ) \sum_m \frac{1}{|\lambda -m^2|} \leq
\frac{\|q\|^2}{\sqrt{|y|}} a_N (y).
$$
On the other hand, by (\ref{4.44})
$$
\sigma_{3,2} \leq \left (\sum_{|s| > |y|^{1/4}} |q(s)|^2 \right )
\cdot 16 a_N (y) \leq 16(\mathcal{E}_{|y|^{1/4}}(q))^2a_N (y),
$$
which completes the proof.
\end{proof}

Lemma \ref{lemBE} is a modification of \cite[Lemma 19]{DM16}. We need
also the following lemma which is a modification of \cite[Lemma
20]{DM16}.

\begin{Lemma}
\label{lemBR} In the above notations, for $bc= Per^\pm $ or $Dir, $
if $h\geq N$ then
\begin{equation}
\label{BR} \sup \{\|K_\lambda V K_\lambda\|_{HS}: \; |Re \,
\lambda|\leq N^2 +N, \;|Im \, \lambda| \geq h \} \lesssim \frac{(\log
h)^{\frac{1}{2}}}{h^{1/4}}\|q\| + \mathcal{E}_{4\sqrt{h}}(q),
\end{equation}
where $q$ is replaced by $\tilde{q}$ if $bc= Dir.$
\end{Lemma}
We omit the proof because it is the same as the proof of Lemma 20 in
\cite{DM16}.
\bigskip

5. We estimate the norm of the operator $B_N $  by using
Lemmas~\ref{lemBE}, \ref{lemBR} and Lemmas~\ref{lem405} and
\ref{lemA} from Appendix. Let $v=Q^\prime,$  and let $q=(q(2k))$ and
$\tilde{q}=(\tilde{q}(m))$ be, respectively, the sequences of
Fourier coefficients of $Q$ about the o.n.b. $\{e^{i2kx}, k \in
\mathbb{Z}\}$ and $\{\sqrt{2}\sin mx, \, m \in \mathbb{N} \}.$
\begin{Proposition}
\label{propBN}

(a) If  $bc=Per^\pm, $ then
\begin{equation}
\label{ec24a} \|B_N \|_{L^1 \to L^\infty} \lesssim
 \left (\frac{\|q\|^2}{N} +
(\mathcal{E}_{\sqrt{N}} (q))^2  \right ) (\log N)^2 +
\int_{N^2}^\infty  \frac{1}{t}  (\mathcal{E}_t (q))^2 dt.
\end{equation}
If $bc=Dir,$ then (\ref{ec24a}) holds with $q$ replaced by
$\tilde{q}.$ \vspace{2mm}

(b)  Suppose $\Omega (t), \; t \in \mathbb{R},$ is a real function
which is even, unbounded and increasing for $x>0.$  If $q \in \ell^2
(\Omega)$  and $bc=Per^\pm,$ then
\begin{equation}
\label{ec24b} \|B_N \|_{L^1 \to L^\infty} \lesssim
 \left (\frac{\|q\|^2}{N} +
\frac{\|q\|^2_{\ell^2 (\Omega)}}{\Omega^2 (\sqrt{N})}
 \right ) (\log N)^2 +
\int_{N^2}^\infty
\frac{\|q\|^2_{\ell^2 (\Omega)}}{t \, \Omega^2 (t)}dt.
\end{equation}
If $bc= Dir $ and $\tilde{q} \in \ell^2 (\Omega),$   then
(\ref{ec24b}) holds with $q$ replaced by $\tilde{q}.$ \vspace{2mm}

(c) If $bc= Per^\pm$ and $q \in \ell^2 (\Omega) $ or $bc=Dir$ and
$\tilde{q} \in \ell^2 (\Omega), $ where $\Omega (k) =
(1+k^2)^{\delta/2}, \; \delta \in (0,1),$ then
\begin{equation}
\label{ecSob} \|B_N \|_{L^1 \to L^\infty} \lesssim  N^{-\delta} (\log
N)^2.
\end{equation}
If $ \, \Omega (k) = (\log (e+k))^{\beta}, \; \beta > 1,  \; $ and
respectively, $bc=Per^\pm $ and $ q \in \ell^2 (\Omega),$ or $bc=Dir$
and $\tilde{q} \in \ell^2 (\Omega), $ then
\begin{equation}
\label{ecLog} \|B_N \|_{L^1 \to L^\infty} \lesssim  (\log
N)^{2-2\beta} \to 0 \quad \text{as} \;\; N \to \infty.
\end{equation}
\end{Proposition}

\begin{proof}
Recall by (\ref{ec22}) and (\ref{ec001}) that
\begin{equation}
\label{ec22*} B_N  = \frac{1}{2 \pi i} \int_{\partial \Pi_N}
\sum_{m=2}^\infty R^0_\lambda (VR^0_\lambda)^m d\lambda,
\end{equation}
where $\; \Pi_N = \{\lambda = x+iy: \; -\omega \leq x \leq N^2 +N, \;
\; |y|\leq h \}.$
 In view of (\ref{4.24}), $R^0_\lambda (VR^0_\lambda)^m
=K_\lambda (K_\lambda VK_\lambda)^m K_\lambda,$ so we have
\begin{equation}
\label{n0} \|R^0_\lambda (VR^0_\lambda)^m \|_{L^1 \to L^\infty} \leq
\|K_\lambda\|_{L^1 \to L^2} \|K_\lambda VK_\lambda\|^m_{L^2 \to L^2}
\|K_\lambda\|_{L^2 \to L^\infty}.
\end{equation}
By (\ref{4.25}) and (\ref{a2}), it follows that
\begin{equation}
\label{n1}
 \|K_\lambda\|^2_{L^1 \to L^2} =\|K_\lambda\|^2_{L^2 \to L^\infty}=
 \sum_k \frac{1}{|\lambda -k^2|} = A(\lambda,1).
\end{equation}
Since the Hilbert-Schmidt norm dominates the $L^2 $-norm, by
(\ref{n0}) and (\ref{n1})  the $\|\cdot\|_{L^1 \to L^\infty}$ norm of
the integrand in (\ref{ec22*}) does not exceed
\begin{equation}
\label{n2}
 \sum_{m=2}^\infty \|R^0_\lambda (VR^0_\lambda)^m \|_{L^1 \to
L^\infty} \leq  S (\lambda),
\end{equation}
where
\begin{equation}
\label{n3} S (\lambda):= A(\lambda, 1) \sum_{m=2}^\infty \|K_\lambda
VK_\lambda\|^m_{HS}.
\end{equation}

The integral in (\ref{ec22*}) does not depend on the choice of the
parameters  $\omega> \omega_0, \; h>h_0 $ in (\ref{ec001}) because
the integrand depends analytically on $\lambda = x+it $ if $x< -
\omega_0, \, |t| > h_0. $ Lemma~\ref{lemBR} implies that if $|Im \,
\lambda|=h$ then $\|K_\lambda V K_\lambda \|_{HS} \leq 1/2 $ for
large enough $h.$ Therefore, in view of (\ref{n2}), (\ref{n3}) and
Lemma~\ref{lem405} (Appendix, formula (\ref{a17}) with $r = 1$),  if
 $N$ is large enough then on the horizontal sides of the rectangle $\Pi_N$
 the norm of the integrand in (\ref{ec22*}) does not exceed
$$
S(\lambda) \leq  A(\lambda, 1) \lesssim  h^{-1/2}, \quad |Im\,
\lambda|=h \geq N^2.
$$
Let $\Lambda_N $ and $\Lambda_N^- $ be the vertical lines
$$
\Lambda_N =  \{\lambda=N^2 +N + iy: \;  y \in \mathbb{R}\}, \quad
\Lambda^-_N =  \{\lambda=- (N^2 +N) + iy: \;   y \in \mathbb{R}\}.
$$
Now, taking $\omega= N^2 +N $ and letting $h \to \infty $ we obtain
(since the integrals on horizontal segments go to zero)  that
\begin{equation}
\label{e1}
B_N  = \frac{1}{2 \pi i} \int_{\Lambda_N} \sum_{m=2}^\infty
R^0_\lambda (VR^0_\lambda)^m d\lambda - \frac{1}{2 \pi i}
\int_{\Lambda^-_N}
 \sum_{m=2}^\infty R^0_\lambda (VR^0_\lambda)^m
d\lambda,
\end{equation}
provided both integrals in (\ref{e1}) converge. Therefore, from
(\ref{n2}) and (\ref{n3}) it follows that
$$ \|B_N \|_{L^1 \to L^\infty} \leq \int_{\Lambda_N \cup \Lambda_N^-}
S(\lambda) dy, \quad \lambda = \pm (N^2+N) =iy.$$

In view of (\ref{4.30}) or (\ref{4.30a}), one can easily see that
$\int_{\Lambda_N^-} S(\lambda)dy \leq \int_{\Lambda_N} S(\lambda)dy$
because $ S(-(N^2+N)+iy) \leq S(N^2+N+iy), $ which implies that
\begin{equation}
\label{72a} \|B_N \|_{L^1 \to L^\infty} \leq 2 \int_{\Lambda_N }
S(\lambda) dy.
\end{equation}
Moreover, for large enough $N$ we have
\begin{equation}
\label{two} \|K_\lambda VK_\lambda\|_{HS} < 1/2  \quad \text{for}
\quad \lambda \in \Lambda_N.
\end{equation}
Indeed, in view of (\ref{Z}) and (\ref{4.34}) in Lemma \ref{lemBE},
we have
$$
\|K_\lambda VK_\lambda\|_{HS} \leq \psi_N (y), \quad \lambda = N^2 +N
+ iy \in \Lambda_N,
$$
with
$$
\psi_N (y) \leq  N^2 \left (  \frac{\|q\|^2}{N}  + 16
(\mathcal{E}_{\sqrt{N}}(q))^2 \right ) b_N (y) + 16
(\mathcal{E}_{4N}(q))^2 a_N (y),
$$
where $a_N (y) $ and $b_N (y)$ are given by (\ref{4.35}). By Lemma
\ref{lemA}, we have that
\begin{equation}
\label{eca} a_N (y) \lesssim \begin{cases}  \frac{\log N}{N}    &
\text{if} \quad   |y| \leq N;\\ \frac{1}{N} \log
(1+\frac{N^2}{|y|}) & \text{if} \quad   N \leq |y| \leq N^2;\\
\frac{1}{\sqrt{|y|}}   & \text{if} \quad  |y| \geq N^2;
\end{cases}
\end{equation}
\begin{equation}
\label{ecb} b_N (y) \lesssim \begin{cases}  \frac{1}{N^2}    &
\text{if} \quad   |y| \leq N;\\ \frac{1}{N|y|}    &  \text{if} \quad
N \leq |y| \leq N^2;\\ \frac{1}{|y|^{3/2}}   & \text{if} \quad  |y|
\geq N^2.
\end{cases}
\end{equation}
Since $\mathcal{E}_{\sqrt{N}}(q) \to 0 $ as $N\to 0, $  by
(\ref{eca}) and (\ref{ecb}) one can easily that $\sup \{\psi_N (y):
y\in \mathbb{R} \} \to 0  $ as $N\to 0, $ which proves (\ref{two}).

From (\ref{two}) it follows that $ \sum_{m=2}^\infty \|K_\lambda
VK_\lambda\|^m_{HS} \leq  \|K_\lambda VK_\lambda\|^2_{HS}$ if $
\lambda \in \Lambda_N.$
 Thus, by (\ref{Z}), (\ref{4.33}) and (\ref{4.35}), we
obtain
\begin{equation}
\label{72b} S(\lambda) \leq a_N (y) \psi_N (y)  \quad \text{for}
\quad \lambda =N^2 +N+iy \in \Lambda_N.
\end{equation}
In view of (\ref{72a}) and (\ref{72b}),
\begin{equation}
\label{ecBN} \|B_N \|_{L^1 \to L^\infty} \leq 2\int_{\mathbb{R}}
a_N (y) \psi_N (y) dy = 2(I_1  +I_2 +I_3 +I_4),
\end{equation}
 where $$ I_1 =
\int_{|y|\leq N}\cdots, \; I_2 = \int_{N<|y|\leq N^2}\cdots, \, I_3 =
\int_{N^2 <|y|\leq N^8}\cdots, \; I_4 = \int_{|y|> N^8}\cdots. $$

 Now we estimate $I_1. $ In view of (\ref{eca}) and
 (\ref{ecb}), if $|y| \leq N $ then $a_N (y) \leq \frac{\log N}{N},  $
and $b_N (y) \leq  \frac{1}{N^2}. $  Therefore, by (\ref{4.34}),
 $$
a_N (y) \psi_N (y) \lesssim \frac{\log N}{N} \left (
\frac{\|q\|^2}{N} +
(\mathcal{E}_{\sqrt{N}} (q))^2  \right )+ (\mathcal{E}_{4N} (q))^2
 \frac{(\log N)^2}{N^2}
 $$
 $$
\lesssim \frac{\log N}{N}
\left (\frac{\|q\|^2}{N} + (\mathcal{E}_{\sqrt{N}} (q))^2  \right ). $$
Therefore, we obtain
\begin{equation}
\label{I1}
 I_1 \leq 2N \cdot \max_{|y| \leq N} a_N (y) \psi_N (y)
\lesssim   \left (\frac{\|q\|^2}{N} +
(\mathcal{E}_{\sqrt{N}} (q))^2  \right ) \log N.
\end{equation}

 Next we estimate $I_2. $ In view of (\ref{eca}) and
 (\ref{ecb}), if $N \leq |y| \leq N^2 $ then $a_N (y) \leq \frac{1}{N}
 \log (1+\frac{N^2}{|y|}),  $
and $b_N (y) \leq  \frac{1}{N|y|}. $  Therefore,  (\ref{4.34}) implies
$$
a_N (y) \psi_N (y) \lesssim
\frac{1}{|y|} \log  \left (1+\frac{N^2}{|y|} \right )
\left (\frac{\|q\|^2}{N} + (\mathcal{E}_{\sqrt{N}} (q))^2  \right )
$$
$$
+ (\mathcal{E}_{4N} (q))^2 \left (\frac{1}{N}
\log \left (1+\frac{N^2}{|y|} \right ) \right )^2
\lesssim \frac{1}{|y|} \log  \left (1+\frac{N^2}{|y|} \right )
\left (\frac{\|q\|^2}{N} + (\mathcal{E}_{\sqrt{N}} (q))^2  \right )
 $$
because $\frac{1}{|y|} \geq \frac{1}{N^2} \log \left
(1+\frac{N^2}{|y|} \right ).$ Now, since $$ \int_N^{N^2} \frac{1}{y}
\log (1+N^2/y) dy \lesssim (\log N) \int_N^{N^2} \frac{1}{y} dy
\lesssim (\log N)^2, $$
 it follows that
\begin{equation}
\label{I2}
 I_2
\lesssim   \left (\frac{\|q\|^2}{N} +
(\mathcal{E}_{\sqrt{N}} (q))^2  \right ) (\log N)^2.
\end{equation}

If $N^2 \leq |y| \leq N^8, $ then by (\ref{4.34}),
(\ref{eca}) and (\ref{ecb})
 it follows that
 $$ a_N (y) \psi_N (y)
\lesssim
\frac{N^2}{y^2} \left (\frac{\|q\|^2}{N} +
(\mathcal{E}_{\sqrt{N}} (q))^2  \right ) +
 \frac{1}{|y|} (\mathcal{E}_{4N} (q))^2. $$

 So, taking into account that $$
\int_{N^2}^{N^8} \frac{1}{y^2} dy < \frac{1}{N^2}, \quad
\int_{N^2}^{N^8} \frac{1}{y} dy = 6 \log N,  $$ we obtain
\begin{equation}
\label{I3}
 I_3 \lesssim  \left (\frac{\|q\|^2}{N} +
(\mathcal{E}_{\sqrt{N}} (q))^2  \right ) \log N.
\end{equation}

To estimate $I_4, $ we  use that  the estimates (\ref{4.36}),
(\ref{eca}) and (\ref{ecb}) imply, for $|y|>N^8,$ that
 $$ a_N (y) \psi_N (y)
\lesssim \frac{N^2}{y^2}  \left (\frac{\|q\|^2}{N} +
(\mathcal{E}_{\sqrt{N}} (q))^2  \right )+
 \frac{1}{|y|} \left (\frac{\|q\|^2}{\sqrt{y}} +
(\mathcal{E}_{|y|^{1/4}} (q))^2  \right ).
$$
Since $\int_{N^8}^{\infty} \frac{1}{y^2} dy = \frac{1}{N^8} $ and
$\int_{N^8}^{\infty} \frac{1}{y^{3/2}} dy  =\frac{2}{N^4}, $
it follows that
\begin{equation}
\label{I4}
 I_4 \lesssim
 \frac{\|q\|^2}{N^4} +\frac{(\mathcal{E}_{\sqrt{N}} (q))^2 }{N^4}+
 \int_{N^8}^\infty \frac{1}{y} (\mathcal{E}_{y^{1/4}} (q))^2 dy.
\end{equation}

In view of (\ref{ecBN}) and (\ref{I1})--(\ref{I4}),
we obtain that (\ref{ec24a}) holds,
which completes the proof (a).

To prove (b), we use that
$$ (\mathcal{E}_M (q))^2 =
\sum_{|s|\geq M} |q(s)|^2 \leq
\frac{1}{(\Omega (M))^2} \sum_{|s|\geq M}
|q(s)|^2 (\Omega (s))^2 \leq  \frac{\|q\|^2_{\Omega}}{(\Omega
(M))^2}, $$ so
\begin{equation}
\label{eclog} \mathcal{E}_{M} (q) \leq
\frac{1}{\Omega (M)} \mathcal{E}_M^{\Omega} (q)
\leq
 \frac{\|q\|_{\Omega}}{\Omega (M)},
 \end{equation}
where
\begin{equation}
\label{ecw} \left (\mathcal{E}_M^{\Omega} (q) \right)^2 =
\sum_{|s|\geq M} |q(s)|^2 (\Omega (s))^2.
\end{equation}
Now (\ref{ec24b}) follows from (\ref{ec24a}) and (\ref{eclog}), which
proves (b). Finally, one can easily see that (c) follows from (b).
\end{proof}
\bigskip

In the proofs of Propositions~\ref{propBN} and
Proposition~\ref{propab} in Section~5, we use Formula (\ref{e1})
(where $B_N$ is written as a difference of two integrals over the
lines $\Lambda_N $ and $\Lambda_N^-$). This representation of $B_N$
is good enough for our proofs.

However, in the context of $L^1$-potentials (see Section 3, Formula
(\ref{3.012})), it is explained (by using simple estimates from
Appendix, Lemmas~\ref{lem405} and \ref{lem407})) that the operator
$B_N$ equals only the integral over $\Lambda_N. $ For singular
potentials, it is more difficult to show that in Formula (\ref{e1})
the integral $\int_{\Lambda^-_N} \sum_{m=2}^\infty R^0_\lambda
(VR^0_\lambda)^m d\lambda =0, $ but it could be done by using
estimates from the proofs of Propositions~\ref{lemBE} and
\ref{propab} . More precisely, the following holds.
\begin{Remark}
\label{rem107} (a) If $v \in H^{-\alpha}, \; \alpha \in (0,1), $ then
\begin{equation}
\label{ee1} B_N  = \frac{1}{2 \pi i} \int_{\Lambda_N}
\sum_{m=2}^\infty R^0_\lambda (VR^0_\lambda)^m d\lambda, \quad N >
N_*.
\end{equation}
where the integral converges in the operator norm $\|\cdot\|_{L^1 \to
L^\infty}.$

(b) If $v \in H^{-1} $ and we consider $B_N $ as an operator from
$L^a $ to $L^b, $  where $1 \leq a < 2 < b \leq \infty $ and
$\frac{1}{a}- \frac{1}{b} <1, $ then (\ref{ee1}) holds and the
integral there converges in the operator norm $\|\cdot\|_{L^a \to
L^b}. $
\end{Remark}

\begin{proof}
For potentials $v \in H^{-\alpha}, \; \alpha \in (0,1), $ Formula
(\ref{e1}) make sense because $\int_{\Lambda_N} S(\lambda) dy $
converge -- see  (\ref{n2}), (\ref{n3}), (\ref{72a}) and the
estimates that follow.
 Using the same argument that leads to Formula
(\ref{e1}) but with $\omega = M^2+M, \; M\in \mathbb{N}, \; M\geq N,$
we obtain
$$
B_N=\frac{1}{2 \pi i} \int_{\Lambda_N} \sum_{m=2}^\infty R^0_\lambda
(VR^0_\lambda)^m d\lambda -\frac{1}{2 \pi i} \int_{\Lambda_M^-}
\sum_{m=2}^\infty R^0_\lambda (VR^0_\lambda)^m d\lambda
$$
Therefore, in view of (\ref{n2}) and (\ref{n3}), we will prove
(\ref{ee1}) if we show that $ \int_{\Lambda_M^-} S(\lambda) dy \to 0
\quad \text{as} \;\; M \to \infty. $ This follows from the Lebesgue
Dominated Convergence Theorem, since by Lemma~\ref{lem407}, Formula
(\ref{a19}),
$$ S(M^2+M+iy) \lesssim A(M^2+M+iy,1) \lesssim (M^2 +y)^{-1/2} \to
0  \;\;  \text{as}  \;\; M \to 0,  $$ and $S(M^2+M+iy) \leq g(y),$
where $g(y) = S(N^2+N+iy) $ is integrable because $\int_{\Lambda_N}
S(\lambda) dy < \infty.$

The proof of (b) is exactly the same, but it is based on inequalities
from the proof of Proposition~\ref{propab}. Therefore, we omit the
details.
\end{proof}

6. Next we estimate the norms of the operator $T_N.$
\begin{Proposition}
\label{propAN}  (a) If $bc=Per^\pm, $ then
\begin{equation}
\label{e24} \|T_N \|_{L^2 \to L^\infty} \lesssim \|q\| \sqrt{H/N} +
\mathcal{E}_H (q) + \mathcal{E}_N (q) (\log N)^{1/2}, \quad 0 < H <
\frac{N}{2}.
\end{equation}
If $bc=Dir,$  then (\ref{e24}) holds with $q$ replaced by
$\tilde{q}.$ \vspace{2mm}

(b) If  $bc=Per^\pm $ and $q \in \ell^2 (\Omega),$ where $\; \Omega
(k) = (1+ |k|^2)^{\delta/2}\;$ with $ \delta \in (0,1), $ and $1\leq
a < 2, $ $a \delta <1,\;$    then
\begin{equation}
\label{e25} \|T_N \|_{L^a \to L^\infty} \lesssim \|q\|
\frac{H^{\frac{1}{a}}}{N^{1/2}} + \mathcal{E}^{\Omega}_{H} (q)
N^{\frac{1}{a}-\delta -1/2}.
\end{equation}
If $bc=Dir$ and $\tilde{q} \in \ell^2 (\Omega),$  then (\ref{e25})
holds with $q$ replaced by $\tilde{q}.$ \vspace{2mm}

(c) If $\Omega (k) = (\log (e+k))^{\beta}, \; \beta \geq 1/2,$  $q
\in \ell^2 (\Omega)$ and $ bc=Per^\pm, $ then
\begin{equation}
\label{e26} \|T_N \|_{L^2 \to L^\infty}  \lesssim \|q\| N^{-1/4} +
(\mathcal{E}^{\Omega}_{\sqrt{N}} (q)) \cdot (\log N)^{1/2-\beta}.
\end{equation}
If $bc=Dir$ and $\tilde{q} \in \ell^2 (\Omega),$  then (\ref{e26})
holds with $q$ replaced by $\tilde{q}.$
\end{Proposition}

\begin{proof}
Suppose $bc = Per^\pm $ and let $(u_k)$ be, respectively,  the
canonical orthonormal basis (\ref{0.011}) or (\ref{0.012}). In view
of (\ref{0010}), (\ref{4.7}) and (\ref{ec21}), if $f = \sum_k f_k u_k
$ is the expansion of $f \in L^2 ([0,\pi]),$ then (\ref{3.21}) gives
$$ T_N  f = \frac{1}{2\pi i}
\int_{\partial \Pi_N} \sum_k f_k \sum_j i
\,\frac{(j-k)q(j-k)}{(\lambda-j^2)(\lambda-k^2)}  u_j (x) d \lambda
= $$
$$ \sum_{|k|\leq N} f_k \sum_{|j|>N}
i \,\frac{(j-k)q(j-k)}{k^2-j^2} u_j (x) +\sum_{|k|> N} f_k
\sum_{|j|\leq N} i\,\frac{(j-k)q(j-k)}{j^2-k^2} u_j (x) $$ By
(\ref{0.011}) or (\ref{0.012}), $|u_j (x)| \leq 1.$ Therefore,  we
have
$$ \|T_N (f) \|_\infty \leq \sum_{|k|\leq N} |f_k| \sum_{|j|>N}
\frac{|q(j-k)|}{|j+k|}  +\sum_{|k|> N} |f_k| \sum_{|j|\leq N}
\frac{|q(j-k)|}{|j+k|}.$$ By the H\"older inequality, it follows that
$$\|T_N (f) \|_\infty  \leq \|(f_k)\|_{\ell^{\tilde{a}}}\,
 (\sigma_1 (a,N))^{1/a}
+\|(f_k)\|_{\ell^{\tilde{a}}} \, (\sigma_2 (a,N))^{1/a},
$$
where $\frac{1}{a}+ \frac{1}{\tilde{a}}=1,$  $
\|(f_k)\|_{\ell^{\tilde{a}}} \leq 2\|f\|_{L^a}$ by the
Young-Haussdorf theorem, and
\begin{equation}
\label{ec35} \sigma_1 (a,N) =\sum_{|k|\leq N} \left ( \sum_{|j|>N}
\frac{|q(j-k)|}{|j+k|} \right )^a, \quad \sigma_2 (a,N) =\sum_{|k|>
N} \left ( \sum_{|j|\leq N} \frac{|q(j-k)|}{|j+k|} \right )^a.
\end{equation}
 Therefore,
\begin{equation}
\label{ec34} \|T_N \|_{L^a \to L^\infty} \lesssim (\sigma_1 (a,N) +
(\sigma_2 (a,N))^{1/a}.
\end{equation}

The situation is similar if $bc=Dir$ and $(u_k) $ is the
corresponding canonical basis (\ref{0.013}). If $f = \sum_{k\in
\mathbb{N}} f_k u_k $ is the expansion of $f \in L^2 ([0,\pi]),$
then by (\ref{d.3}), (\ref{d.10}), (\ref{ec21}) and (\ref{3.21}) we
obtain
$$ T_N  f = \frac{1}{2\pi i}
\int_{\partial \Pi_N} \sum_k f_k \sum_j
\frac{|j-k|\,\tilde{q}(|j-k|)-
(j+k)\,\tilde{q}(j+k)}{\sqrt{2}(\lambda-j^2)(\lambda-k^2)} u_j (x) d
\lambda =
$$
$$ -\sum_{1\leq k\leq N} f_k \sum_{j>N}
\frac{\tilde{q}(j-k)}{j+k} \frac{1}{\sqrt{2}}u_j (x)  + \sum_{1\leq
k\leq N} f_k \sum_{j>N} \frac{\tilde{q}(j+k)}{j-k}
\frac{1}{\sqrt{2}}u_j (x)
$$
$$
- \sum_{k> N} f_k \sum_{1\leq j \leq N} \frac{\tilde{q}(k-j)}{j+k}
\frac{1}{\sqrt{2}}u_j (x) + \sum_{k> N} f_k \sum_{1\leq j \leq N}
\frac{\tilde{q}(j+k)}{k-j}\frac{1}{\sqrt{2}}u_j (x).$$ By
(\ref{0.051}), $|u_j (x)|\leq \sqrt{2}, $ so using the H\"older
inequality as above we obtain (\ref{ec34}) holds with
\begin{equation}
\label{ec35d} \sigma_1 (a,N) = \sum_{k=-N}^N \left ( \sum_{j>N}
\frac{|\tilde{q}(j-k)|}{j+k} \right )^a, \quad \sigma_2 (a,N)
=\sum_{k> N}  \left (\sum_{j=-N}^N  \frac{|\tilde{q}(k-j)|}{k+j}
\right )^a.
\end{equation}
In view of (\ref{ec35}) and (\ref{ec35d}), if we set $q(k) = 0 $ for
$k \in 2\mathbb{Z}+1 $ in the case $bc= Per^\pm,$ and $q(k)
=\tilde{q} (|k|) $ in the case $bc= Dir,$ and define $\sigma_1,
\,\sigma_2 $ by (\ref{ec35}) with $j,k\in \mathbb{Z},$ then
(\ref{ec34}) holds in all three cases $bc=Per^\pm,\, Dir. $ Next we
estimate $\sigma_1 $ and $\sigma_2 $ in terms of remainders
$\mathcal{E}_M (q).$
\begin{Lemma}
\label{lemE} If $q\in \ell^2 ( \mathbb{Z}), $ then
\begin{equation}
\label{e71} \sigma_1  (a,N) \lesssim \frac{1}{N^{a/2}} \sum_{0 \leq
k\leq N} (\mathcal{E}_{N+1-k} (q))^a +
\begin{cases}
(\mathcal{E}_N (q))^a N^{1-\frac{a}{2}}, &  1 \leq a <2,\\
(\mathcal{E}_N (q))^2 \log N, &   a=2,
\end{cases}
\end{equation}
\begin{equation}
\label{e72} \sigma_2  (a,N) \lesssim \sum_{k> N} (\mathcal{E}_{k-N}
(q))^a \left ( \frac{1}{k} -\frac{1}{N+k} \right )^{a/2} +
\begin{cases}
(\mathcal{E}_N (q))^a N^{1-\frac{a}{2}}, &  1 \leq a <2,\\
(\mathcal{E}_N (q))^2 \log N, &   a=2.
\end{cases}
\end{equation}
\end{Lemma}

\begin{proof}
Changing, if negative, $k$ with $-k$ and $j$ with $-j$ we obtain
$$ \sigma_1 (a, N) \lesssim \sum_{0 \leq k\leq N} \left ( \sum_{j>N}
\frac{|q(j-k)|}{j+k} \right )^a + \sum_{0 \leq k\leq N}\left
(\sum_{j>N} \frac{|q(-j+k)|}{j+k} \right )^a $$ $$ + \sum_{0 \leq
k\leq N} \left ( \sum_{j>N} \frac{|q(j+k)|}{j-k} \right )^a +\sum_{0
\leq k\leq N} \left (\sum_{j>N} \frac{|q(-j-k)|}{j-k} \right )^a. $$
By the Cauchy inequality, it follows that $$ \sigma_1 (a,N) \lesssim
\sum_{0 \leq k\leq N} \left ( \sum_{j>N} (|q(j-k)|^2 +|q(-j+k)|^2
)\right )^{a/2} \left (\sum_{j>N} \frac{1}{(j+k)^2} \right )^{a/2}
$$ $$ + \sum_{0 \leq k\leq N} \left ( \sum_{j>N} (|q(j+k)|^2
+|q(-j-k)|^2 )\right )^{a/2} \left (\sum_{j>N} \frac{1}{(j-k)^2}
\right )^{a/2}. $$ Therefore, $$ \sigma_1 (a,N)\lesssim \sigma_{1,1}
(a, N) +\sigma_{1,2} (a, N), $$ where $$ \sigma_{1,1}  = \sum_{0 \leq
k\leq N} (\mathcal{E}_{N+1-k} (q))^a \frac{1}{(N+k)^{a/2}} \lesssim
\frac{1}{N^{a/2}} \sum_{0 \leq k\leq N} (\mathcal{E}_{N+1-k} (q))^a,
$$
 $$ \sigma_{1,2}
=\sum_{0 \leq k\leq N} (\mathcal{E}_{N+1+k} (q))^a
\frac{1}{(N+1-k)^{a/2}} \lesssim
\begin{cases}
(\mathcal{E}_N (q))^aN^{1-\frac{a}{2}}, &  1 \leq a <2,\\
(\mathcal{E}_N (q))^2 \log N, &   a=2,
\end{cases}
$$
because
$$
\sum_{s>M} \frac{1}{s^2} \leq \frac{1}{M}, \quad \sum_{s=1}^N
\frac{1}{s^{a/2}}  \lesssim \int_1^N x^{-\frac{a}{2}} dx \lesssim
\begin{cases}
N^{1-\frac{a}{2}}, &  1 \leq a <2,\\ \log N, &   a=2.
\end{cases}
$$
Thus, (\ref{e71}) holds.

Next we estimate $ \sigma_2 (a, N). $ As for $\sigma_1 (a,N), $ we
obtain
 $$ \sigma_2  \lesssim \sum_{ k>
N} \left ( \sum_{0 \leq j\leq N} (|q(j-k)|^2 +|q(-j+k)|^2 )\right
)^{a/2} \left (\sum_{0 \leq j\leq N} \frac{1}{(j+k)^2} \right
)^{a/2}
$$ $$ + \sum_{ k > N} \left ( \sum_{0 \leq j\leq N} (|q(j+k)|^2
+|q(-j-k)|^2 )\right )^{a/2} \left (\sum_{0 \leq j\leq N}
\frac{1}{(k-j)^2} \right )^{a/2}. $$ Therefore, $$ \sigma_2 (a,
N)\lesssim \sigma_{2,1} (a,N) +\sigma_{2,2} (a,N),
$$ where $$ \sigma_{2,1}  =  \sum_{k> N} (\mathcal{E}_{k-N}
(q))^a \left ( \frac{1}{k} -\frac{1}{N+k} \right )^{a/2}, $$ $$
\sigma_{2,2} =\sum_{k> N} (\mathcal{E}_{k} (q))^a \left (
\frac{1}{k-N} -\frac{1}{k} \right )^{a/2} \lesssim \begin{cases}
(\mathcal{E}_{N} (q))^a N^{1-a/2}, &  1<a<2, \\
(\mathcal{E}_{N} (q))^a \log N, &   a=2,
\end{cases}
$$
because $\sum_{M_1}^{M_2} \frac{1}{s^2} \lesssim \frac{1}{M_1} -
\frac{1}{M_2}, $
$$
\sum_{k> N}  \left ( \frac{1}{k-N} -\frac{1}{k} \right
)^{\frac{a}{2}} = N^{\frac{a}{2}} \left (\sum_{s=1}^N
\frac{1}{s^{\frac{a}{2}}(N+s)^{\frac{a}{2}}} +\sum_{s=N+1}^\infty
\frac{1}{s^{\frac{a}{2}}(N+s)^{\frac{a}{2}}}\right )
$$
$$
\lesssim N^{\frac{a}{2}} \left (\sum_{s=1}^N
\frac{1}{s^{\frac{a}{2}}N^{\frac{a}{2}}} +\sum_{s=N+1}^\infty
\frac{1}{s^a}\right ) \lesssim N^{1-a/2} \quad \text{if}  \;\;
1<a<2,
$$
and
$$
\sum_{k> N}  \left ( \frac{1}{k-N} -\frac{1}{k} \right ) =
\lim_{M\to \infty} \sum_1^M \left (\frac{1}{s} -\frac{1}{s+N} \right
)=\sum_1^N \frac{1}{s} \lesssim \log N.
$$
Hence, (\ref{e72}) follows.

\end{proof}

The following lemma proves (\ref{e24}) and (\ref{e26}).

\begin{Lemma}
\label{lemL} In the above notations, if $q \in \ell^2 (\mathbb{Z}) $
and $ H \in (0, N/2), $ then
\begin{equation}
\label{e200} \sigma_1  (2,N)+ \sigma_2  (2,N) \lesssim \|q\|^2
\frac{H}{N} + (\mathcal{E}_{H} (q))^2 + (\mathcal{E}_{N} (q))^2
\cdot \log N.
\end{equation}
Moreover, if $q \in \ell^2 (\Omega, \mathbb{Z}) $ with $\Omega (k) =
(\log (e+|k|))^\beta, $   $ \beta > 1/2,$  then
\begin{equation}
\label{e201} \sigma_1  (2,N)+ \sigma_2  (2,N) \lesssim \|q\|^2
N^{-1/2} + (\mathcal{E}^{\Omega}_{\sqrt{N}} (q))^2 \cdot (\log
N)^{1-2\beta}.
\end{equation}

\end{Lemma}

\begin{proof}
Indeed, (\ref{e200}) follows from (\ref{e71}) and (\ref{e72})
because
$$
\sum_{k=0}^N (\mathcal{E}_{N+1-k} (q))^2 \leq
\sum_{k=0}^{N-H}(\mathcal{E}_{H} (q))^2+ \sum_{N-H}^N \|q\|^2 \leq N
\cdot(\mathcal{E}_{H} (q))^2+ H \|q\|^2,
$$
and
$$
\sum_{k>N} (\mathcal{E}_{k-N} (q))^2 \left (\frac{1}{k}-
\frac{1}{k+N} \right )\leq \sum_{k=N+1}^{N+H} \|q\|^2 \frac{1}{k}+
\sum_{k>N+H} (\mathcal{E}_{H} (q))^2 \left (\frac{1}{k}-
\frac{1}{k+N} \right )
$$
$$
\leq  \|q\|^2 \frac{H}{N} + (\mathcal{E}_{H} (q))^2
\sum_{k=N}^\infty \frac{N}{k(k+1)}=\|q\|^2 \frac{H}{N} +
(\mathcal{E}_{H} (q))^2.
$$

If $q \in \ell^2(\Omega) $ with $\Omega (k) =(\log (e+|k|))^\beta, $
then by (\ref{eclog}) $\mathcal{E}_{M} (q) \leq
\frac{\mathcal{E}^\Omega_{M} (q)}{(\log (e+M))^\beta}. $ Therefore,
(\ref{e200}) with $H=\sqrt{N} $ implies (\ref{e201}).
\end{proof}

Now (\ref{e24}) and (\ref{e26}) follow immediately from (\ref{ec34})
and, respectively, (\ref{e200}) and (\ref{e201}).

\begin{Lemma}
\label{lemd}
 If $1 \leq a < 2 $  and $q \in \ell^2 (\Omega) $ with $\Omega
(k) = 1+ |k|^\delta, $  $ \delta > 0 $ then
\begin{equation}
\label{e76} \sigma_1 (a,N)+\sigma_2 (a, N) \lesssim \|q\|^a
\frac{H}{N^{a/2}} +(\mathcal{E}^{\Omega}_{H} (q))^a N^{1-\delta a
-\frac{a}{2}}, \quad 0 <H <\frac{N}{2}.
\end{equation}
\end{Lemma}

\begin{proof}
In view of (\ref{eclog}), $\mathcal{E}_{M} (q)\leq
\mathcal{E}^\Omega_{M} (q)/\Omega (M)=\mathcal{E}^\Omega_{M}
(q)/M^\delta,$ so (\ref{e71}) implies
$$
\sigma_1 (a,N) \lesssim N^{-\frac{a}{2}}\sum_{k=0}^{N-H}
\frac{(\mathcal{E}^\Omega_{N+1-k}(q))^a}{(N+1-k)^{a\delta}}+N^{-\frac{a}{2}}
\sum_{N-H+1}^{N} \|q\|^2
  +   \frac{(\mathcal{E}^\Omega_{N} (q))^a}{N^{a\delta}} N^{1-\frac{a}{2}}.$$
Therefore, taking into account that $\mathcal{E}^\Omega_{N+1-k}
(q)\leq \mathcal{E}^\Omega_{H} (q) $ for $ 0\leq k \leq N-H, $ we
obtain
\begin{equation}
\label{es1} \sigma_1 (a,N) \lesssim \|q\|^a \cdot \frac{H}{N^{a/2}}
+(\mathcal{E}^\Omega_{H} (q))^a N^{1-a\delta-\frac{a}{2}}
\end{equation}
because $\sum_{k=0}^{N-H} \frac{1}{(N+1-k)^{a\delta}}\lesssim
\sum_{s=1}^N \frac{1}{s^{a\delta}}\lesssim N^{1-a\delta}.$

Next we estimate $\sigma_2$ in an analogous way. From (\ref{e72}) it
follows that
$$
\sigma_2 (N) =\sum_{N+1}^{N+H} \|q\|^2
 \frac{1}{N^{a/2}}  +  \sum_{N+H+1}^\infty
(\mathcal{E}_{k-N} (q))^a \left ( \frac{1}{k} -\frac{1}{N+k} \right
)^{\frac{a}{2}} + (\mathcal{E}_{N} (q))^a N^{1-\frac{a}{2}}$$
$$ \leq \|q\|^a \cdot \frac{H}{N^{a/2}} +
(\mathcal{E}^\Omega_{H} (q))^a \sum_{N+H+1}^\infty
\frac{1}{(k-N)^{a\delta}}\left ( \frac{1}{k} -\frac{1}{N+k} \right
)^{\frac{a}{2}} + (\mathcal{E}^\Omega_{N} (q))^a
N^{1-a\delta+\frac{a}{2}}.
$$

Since $\mathcal{E}^\Omega_{N}(q) \leq \mathcal{E}^\Omega_{H}(q)$ and
$$
\sum_{N+H+1}^\infty \frac{1}{(k-N)^{a\delta}}\left ( \frac{1}{k}
-\frac{1}{N+k} \right )^{\frac{a}{2}} =
\sum_{N+1}^{2N}\frac{1}{(k-N)^{a\delta}}
\frac{N^{\frac{a}{2}}}{k^{\frac{a}{2}}(k+N)^{\frac{a}{2}}}
$$
$$
+\sum_{2N+1}^{\infty}\frac{1}{(k-N)^{a\delta}}
\frac{N^{\frac{a}{2}}}{k^{\frac{a}{2}}(k+N)^{\frac{a}{2}}} \lesssim
N^{-\frac{a}{2}} \sum_{s=1}^N \frac{1}{s^{a\delta}}+
N^{\frac{a}{2}}\sum_{N}^\infty \frac{1}{s^{a\delta +a}}\lesssim
N^{1-a\delta-\frac{a}{2}},
$$
we obtain
\begin{equation}
\label{es2} \sigma_2 (a,N) \lesssim \|q\|^a \cdot \frac{H}{N^{a/2}}
+(\mathcal{E}^\Omega_{H} (q))^a N^{1-a\delta-\frac{a}{2}}.
\end{equation}
Now, (\ref{es1}) and (\ref{es2}) imply (\ref{e76}).
\end{proof}
Finally, (\ref{ec34}) and (\ref{e76}) imply (\ref{e25}), which
completes the proof of Proposition~\ref{propAN}.
\end{proof}

\section{The case $v \in H^{-1}_{per},$
 $S_N -S_N^0:L^a \to L^b, \; b <\infty.$}

1. Our main result in this section is the following.
\begin{Theorem}
\label{thmab} Suppose
 $S_N, \, S_N^0 $ are the spectral
projections defined by (\ref{ec20}) for the Hill operators $L_{bc}
(v) $ and $L_{bc}^0 $ subject to the boundary conditions $bc =
Per^\pm $ or $ Dir. $ Let $v \in H^{-1}_{per},  \;  v = Q^\prime,\;
Q \in L^2 ([0,\pi]), $ and let $q= (q(k))_{k\in 2\mathbb{Z}}$ and
$\tilde{q}=(\tilde{q}(m))_{m\in \mathbb{N}}$ be, respectively, the
sequences of the Fourier coefficients of $Q $   about the o.n.b.
$\{e^{ikx}, \, k \in 2\mathbb{Z}\} $ and $\{\sqrt{2} \sin mx, \, m
\in \mathbb{N}\}.$  If
\begin{equation}
\label{ab1} 1< a \leq b < \infty \quad \text{with} \quad \delta:=
1/2 - (1/a -1/b) >0,
\end{equation}
then
\begin{equation}
\label{ab2} \|S_N - S_N^0: L^a \to L^b \| \lesssim N^{-\tau} +
\begin{cases}
\mathcal{E}_N(q)  & \text{if} \;\;  bc=Per^\pm, \\
\mathcal{E}_N(\tilde{q}) & \text{if} \;\; bc=Dir,
\end{cases}
\end{equation}
where $\tau= \delta $ in the case $ 1 < a < 2 < b < \infty, $ and
otherwise one may take any $\tau $ such that
$$\tau <  \begin{cases}  1- 1/a    &   \text{if}
\;\; 1 < a \leq b \leq 2,\\
1/b   &  \text{if} \;\;  2 \leq  a \leq b < \infty.
\end{cases}  $$
 \end{Theorem}

{\em Remark}. This theorem (quoted as Proposition 16 in \cite{DM28})
is an important element in the
 proof of our Criterion for basisness in $L^p, \; p \neq 2,$
 of the system of root functions
 in the case of Hill
operators with  singular potentials.

\begin{proof}
By (\ref{ec2a}),  $S_N - S_N^0 = T_N + B_N, $ where $T_N $ and $B_N $
are the operators defined by (\ref{ec21}) and (\ref{ec22}). If $1< a
< 2 < b < \infty, $ then Propositions~\ref{propTab} and \ref{propab}
below imply (\ref{ab2}) with $\tau = \delta. $

If (\ref{ab1}) holds but  $ a < 2 < b $ fails, then $$ \text{either}
\;\;(i) \;\;    1<a \leq b \leq 2, \quad \text{or} \;\; (ii) \;\; 2
\leq a \leq b < \infty. $$  We set, respectively, $a_1 =
a-\varepsilon, \, b_1 = 2+ \varepsilon $ in the case (i), and $a_1
=2-\varepsilon, \, b_1 = b+ \varepsilon $ in the case (ii). Then, for
small enough $\varepsilon >0,$ we have
$$1<a_1 < 2 < b_1 <\infty,
\quad \delta_1 =1/2-(1/a_1 -1/b_1) >0.
$$
Since $ \|S_N - S_N^0: L^a \to L^b \| \leq  \|S_N - S_N^0: L^{a_1}
\to L^{b_1} \|, $  it follows that
$$\|S_N - S_N^0: L^a \to L^b \|
\lesssim N^{-\delta_1} + \begin{cases}
\mathcal{E}_N(q)  & \text{if} \;\;  bc=Per^\pm, \\
\mathcal{E}_N(\tilde{q}) & \text{if} \;\; bc=Dir,
\end{cases}
$$
so (\ref{ab2}) holds with $ \tau = \delta_1. $ But in both cases
$\delta_1=\delta_1 (\varepsilon)$ is a monotone decreasing function
of $\varepsilon $ such that
 $\lim_{\varepsilon \to 0} \delta_1
(\varepsilon) =
\begin{cases}  1- 1/a    &   \text{if}
\;\; 1 < a \leq b \leq 2,\\
1/b   &  \text{if} \;\;  2 \leq  a \leq b < \infty.
\end{cases}
$ This completes the proof up to Propositions~\ref{propTab} and
\ref{propab} below.
\end{proof}
\bigskip

2. Next we estimate the norms $\|T_N: L^a \to L^b\|. $
\begin{Proposition}
\label{propTab} If $v \in H^{-1}_{per}, $ then for $1 < a < 2 <b <
\infty $ with
\begin{equation}
\label{3.201} \delta = 1/2 - (1/a - 1/b)>0
\end{equation}
\begin{equation}
\label{3.202} \|T_N: L^a \to L^b\| \lesssim N^{-\delta}
+\begin{cases}
\mathcal{E}_N(q)  & \text{if} \;\;  bc=Per^\pm, \\
\mathcal{E}_N(\tilde{q}) & \text{if} \;\; bc=Dir.
\end{cases}
\end{equation}
\end{Proposition}

\begin{proof}
As in Section 3.3, we obtain the matrix representation of the
operator $T_N $ {\em after integration over} $\partial \Pi_N. $ If
$T_{mk}$ is its matrix representation with respect to the basis
$\{u_k, \, k \in \Gamma_{bc}\}$ of eigenfunctions of the free
operator $L^0_{bc}, $ then $T_{mk}= 0 $ for $(m,k) \not \in X,$
where $X= X(N) $ is defined in (\ref{3.23}) or (\ref{3.23d}).

By the Hausdorf-Young theorems,
\begin{equation}
\label{3.203}
 \|T_N: L^a \to L^b\| \leq 4 \tilde{\tau}, \quad
\tilde{\tau} = \|\tilde{T}_N: \ell^\alpha \to \ell^\beta\|
\end{equation}
where
\begin{equation}
\label{3.204} 1/a + 1/\alpha =1, \quad 1/b + 1/\beta = 1
\end{equation}
and the operator $\tilde{T}_N $ is defined by its matrix,
respectively given by (\ref{3.21a}) if $bc=Per^\pm $ and
(\ref{3.21d}) if $bc=Dir.$ Further we provide details only in the
case $bc = Per^\pm $ because the proof for $bc=Dir $ is the same.

By duality
\begin{equation}
\label{3.205} \tilde{\tau} \leq \sup \left \{ \sum_{(k,m)\in X_N}
\frac{|V(m-k)|}{|m^2 - k^2|} |f(k)| |g(m)|: \; \|f|\ell^\alpha \|=
1, \; \|g|\ell^b \|=1 \right \}.
\end{equation}
Therefore, in view of (\ref{0010}) we need to evaluate
\begin{equation}
\label{3.206} \tau (f,g) = \sum_{(k,m)\in X_N}  \frac{|q(m-k)|}{|m +
k|}
 |f(k)| |g(m)|.
\end{equation}
We set
\begin{equation}
\label{3.207} \Delta_N = \{(k,m) \in X_N: \; |m-k|\leq N  \}, \quad
\Delta^c_N = X_N \setminus \Delta_N,
\end{equation}
and analyze the corresponding partial sums of $\tau(f,g).$

Since $|m+k| \geq N $ on $\Delta_N, $ it follows that
\begin{equation}
\label{3.209} \sum_{(k,m)\in \Delta_N} \; \leq \frac{4}{N}
\sum_{j=1}^N |q(j)| \left (\sum_{\ell_j} |f(k)| |g(m) \right ),
\end{equation}
where
\begin{equation}
\label{3.210} \ell_j = \{(k,m)\in X_N: \;  |m-k| = j\}.
\end{equation}
If
\begin{equation}
\label{3.211} \frac{1}{\alpha} + \frac{1}{b} +\frac{1}{\gamma} = 1
\end{equation}
(so, by (\ref{3.201}), $\frac{1}{\gamma}= \frac{1}{a}  -\frac{1}{b}
<\frac{1}{2} $), then by the triple H\"older inequality
\begin{equation}
\label{3.212} \sum_{\ell_j} |f(k)| \cdot |g(m)|\cdot 1 \leq
\|f|\ell^\alpha \| \cdot \|g|\ell^b \| \cdot (card\,
\ell_j)^{1/\gamma} \leq (4j)^{1/\gamma},
\end{equation}
and by the Cauchy inequality
\begin{equation}
\label{3.213} \sum_1^N |q(j)| \, j^{1/\gamma} \leq \|q\| \left
(\sum_1^N j^{2/\gamma} \right )^{1/2} \sim \|q\| \cdot N^{1/\gamma +
1/2}.
\end{equation}
With extra-factor $1/N $ in (\ref{3.209}) these inequalities imply
that
\begin{equation}
\label{3.214} \sum_{\Delta_N} \leq C(\gamma) N^{-\delta}, \quad
\delta = \frac{1}{2}- \frac{1}{\gamma}.
\end{equation}
To estimate $\sum_{\Delta_N^c} $ we choose positive $p, q , r $ with
$p+q+r=1 $ in the following way:
\begin{equation}
\label{3.215} p = t/2, \quad q= t(1/a - 1/2), \quad r = t(1/2 - 1/b)
\end{equation}
with
\begin{equation}
\label{3.216} 1/2 < 1/t = (1/a - 1/b) +1/2 <1.
\end{equation}
Then
\begin{equation}
\label{3.218} \sum_{\Delta_N^c} \frac{|q(m-k)|}{|m+k|^p} \cdot
\frac{|f(k))|}{|m+k|^p} \cdot  \frac{|g(m)|}{|m+k|^r}
\end{equation}
$$
\leq \left ( \sum_{\Delta_N^c} \frac{|q(m-k)|^2}{|m+k|^{2p}} \right
)^{1/2} \left ( \sum_{\Delta_N^c} \frac{|f(k)|^2}{|m+k|^{2q}} \cdot
\frac{|g(m)|^2}{|m+k|^{2r}} \right )^{1/2}.
$$
With $|m-k| >N $  on  $\Delta_N^c$ the first factor in the
right-hand side of (\ref{3.218}) does not exceed
\begin{equation}
\label{3.219} 8 \mathcal{E}_N (q) \cdot \left ( \sum_1^\infty
\frac{1}{j^{2p} } \right )^{1/2} =C(p) \mathcal{E}_N (q)  < \infty.
\end{equation}
In the second factor we want to make $k$ and $m$ independent; we can
achieve this on four subsets of $\Delta_N^c $ separately, where
\begin{equation}
\label{3.220} \Delta_N^c =F^+_1 \cup F^-_1 \cup F^+_2 \cup F^-_2 \
\end{equation}
with
\begin{equation}
\label{3.221} F^\pm_j = \{(k_1, k_2) \in  \Delta_N^c: \; |k_j | \leq
N,  \; \pm k_{j^\prime} >0\}.
\end{equation}
For $(k,m) \in F^+_1 $ we have $|k| \leq N $ and $m \geq N+1;$ then
either $|m+k| = m+k \geq m-N \geq 1$ or $m+k \geq N+1 +k \geq N+1.$
Therefore,
$$
\sum_{F^+_1} \frac{|f(k)|^2}{|m+k|^{2q}} \cdot
\frac{|g(m)|^2}{|m+k|^{2r}} \leq \sum_{F^+_1} \frac{|f(k)|^2}{|N+1+k
|^{2q}} \cdot \frac{|g(m)|^2}{|m-N|^{2r}}
$$
$$
\leq
 \sum_{i=1}^\infty \frac{|f(-1-N+i)|^2}{i^{2q}}  \cdot
 \sum_{j=1}^\infty \frac{|g(N+j)|^2}{j^{2r}} :=
 \mathcal{F}  \cdot  \mathcal{G}.
$$
 Each of these two factors $\mathcal{F}, \mathcal{G}$
 is estimated by the H\"older inequality,
 respectively with parameters
 $\alpha/2, \; \tilde{\alpha}  $ and $b/2, \tilde{b},$
 i.e.,
 $$
 \frac{2}{\alpha} + \frac{1}{\tilde{\alpha}}=1, \quad
 \frac{2}{b} + \frac{1}{\tilde{b}}=1.
 $$
 This choice, together with (\ref{3.215}) and (\ref{3.216})
 guarantees that
 $$
 2q \tilde{\alpha} >1, \quad 2r \tilde{b} >1,
 $$
 so the first factor does not exceed
 $$
 \mathcal{F} \leq \left (  \sum_j   |f (-1-N+i)|^\alpha
\right )^{2/\alpha} \cdot \left (\sum_i (1/i)^{2q\tilde{\alpha}}
\right )^{1/\tilde{\alpha}} <\infty.
 $$
The same argument with $2r \tilde{b} >1 $ shows that
$$
\mathcal{G} \leq C(g) \cdot \left (\sum_j (1/j)^{2r\tilde{b}} \right
)^{1/\tilde{b}}<\infty.
$$
The other sums over $ F_j^\pm $ could be estimated in an analogous
way. This shows that the sum in (\ref{3.218}) does not exceed
$C(a,b)  \cdot \mathcal{E}_N (q),$
 so together with (\ref{3.214}) we obtain for the form
 $\tau (f,g) $ that
 $$
 \tau (f,g)  \leq C(a,b) \left
 ( N^{-\delta} + \mathcal{E}_N (q) \right ).
 $$
This implies (\ref{3.202}), which completes the proof.
\end{proof}
\bigskip

3. Finally, we estimate the norms $\|B_N:\, L^a \to L^b\|.$
\begin{Proposition}
\label{propab} If
\begin{equation}
\label{ec100} 1 \leq  a < 2 < b \leq  \infty, \quad  1/a - 1/b < 1,
\end{equation}
then
\begin{equation}
\label{ec101} \|B_N:\, L^a \to L^b\| \lesssim
 \frac{\|q\|^2}{N} +\begin{cases}
(\mathcal{E}_{\sqrt{N}} (q))^2 &  \text{if} \;\; bc=Per^\pm,\\
(\mathcal{E}_{\sqrt{N}} (\tilde{q}))^2 &  \text{if} \;\; bc=Dir.
 \end{cases}
\end{equation}
\end{Proposition}

\begin{proof}
We provide details only in the case $bc = Per^\pm $ because the proof
for $bc=Dir $ is the same.

By (\ref{e1}), as in the proof of Proposition~\ref{propBN}, it
follows that
$$ \|B_N: L^a \to L^b\| \leq \int_{\Lambda_N \cup
\Lambda^-_N} \sum_{m=2}^\infty   \|R^0_\lambda (VR^0_\lambda)^m: L^a
\to L^b \|  dy.$$ By (\ref{4.24}), $R^0_\lambda (VR^0_\lambda)^m
=K_\lambda (K_\lambda VK_\lambda)^m K_\lambda,$ so we have
\begin{equation}
\label{ec81} \|R^0_\lambda (VR^0_\lambda)^m \|_{L^a \to L^b} \leq
\|K_\lambda\|_{L^a \to L^2} \|K_\lambda VK_\lambda\|^m_{L^2 \to L^2}
\|K_\lambda\|_{L^2 \to L^b}.
\end{equation}
Recall that $K_\lambda $ is defined by (\ref{4.25}) as a multiplier
operator in the sequence spaces of Fourier coefficients. If $f\in L^a
$ and $(f_k)$ is its sequence of Fourier coefficients about $\{u_k
(x)\}$ -- one of our canonical o.n.b. (\ref{0.011}), (\ref{0.012}) --
then $K_\lambda f = \sum_k \frac{1}{(\lambda -k^2)^{1/2}} f_k u_k
(x). $ By Hausdorff-Young Theorem $(f_k) \in \ell^\alpha $ with $
\frac{1}{a} + \frac{1}{\alpha} =1, $ and $\|(f_k)\|_{\ell^\alpha}
\leq 2\|f\|_{L^a}.$ The H\"older inequality implies (compare with
(\ref{mult})) that
$$
\|K_\lambda: L^a \to L^2\| \lesssim  \left ( \sum_k \frac{1}{|\lambda
- k^2|^{\frac{a}{2-a}}} \right )^{\frac{2-a}{2a}}, \quad  1\leq a <2.
$$
By duality argument, $\|K_\lambda: L^2 \to L^b\|=\|K_\lambda: L^\beta
\to L^2\|, $ where $\frac{1}{\beta} + \frac{1}{b} =1, $ so it follows
that
$$
\|K_\lambda: L^2 \to L^b\| \lesssim  \left ( \sum_k \frac{1}{|\lambda
- k^2|^{\frac{\beta}{2-\beta}}} \right )^{\frac{2-\beta}{2\beta}},
\quad 2< b \leq \infty, \;\; \frac{1}{\beta} =1- \frac{1}{b}.
$$
Therefore, in view of (\ref{a2}) we have
$$
\|K_\lambda\|_{L^a \to L^2} \lesssim A^{1/2} \left
(\lambda,\frac{a}{2-a} \right), \quad \|K_\lambda\|_{L^2 \to L^b}
\lesssim A^{1/2} \left (\lambda, \frac{\beta}{2-\beta} \right ).
$$

Since the Hilbert-Schmidt norm dominates the $L^2 $-norm,
(\ref{ec81}) and the above formulas imply that
$$ \sum_{m=2}^\infty \|R^0_\lambda (VR^0_\lambda)^m \|_{L^a \to
L^b} \leq  S_1 (\lambda), $$ where
$$S_1 (\lambda):=
 A^{1/2} \left
(\lambda,\frac{a}{2-a} \right) A^{1/2} \left (\lambda,
\frac{\beta}{2-\beta} \right ) \sum_{m=2}^\infty \|K_\lambda
VK_\lambda\|^m_{HS}.
$$

As in the proof of Proposition~\ref{propBN}, by (\ref{4.30}) one can
easily see that $\int_{\Lambda_N^-} S_1 (\lambda)dy \leq
\int_{\Lambda_N} S_1 (\lambda)dy.$ Therefore,
 $$\|B_N \|_{L^a \to L^b} \lesssim
\int_{\Lambda_N } S_1 (\lambda) dy. $$

In view of (\ref{two}),
 for large enough $N$ we have
$\|K_\lambda VK_\lambda\|_{HS} < 1/2 $  for $\lambda \in \Lambda_N,
$ so
$$
\sum_{m=2}^\infty \|K_\lambda VK_\lambda\|^m_{HS} \leq  2\|K_\lambda
VK_\lambda\|^2_{HS}, \quad N \geq N_*.
$$
Thus, by (\ref{Z}) and (\ref{4.33}),  we obtain $S_1 (\lambda) \leq
\Phi_N (y) $ with
\begin{equation}
\label{e4}  \Phi_N (y) := A^{\frac{1}{2}} \left
(\lambda,\frac{a}{2-a} \right) A^{\frac{1}{2}} \left (\lambda,
\frac{\beta}{2-\beta} \right ) \, \psi_N (y), \quad \lambda =N^2
+N+iy.
\end{equation}
In view of the above formulas,
\begin{equation}
\label{ecBNa} \|B_N \|_{L^a \to L^b} \lesssim \int_{\mathbb{R}}
 \Phi_N (y) dy \lesssim I_1  +I_2 +I_3, \quad N \geq N_*,
\end{equation}
 where $$ I_1 =
\int_{|y|\leq N} \Phi_N (y) dy, \quad  I_2 = \int_{N\leq |y|\leq
N^2}\Phi_N (y) dy, \quad I_3 = \int_{|y|\geq N^2 } \Phi_N (y) dy.$$
Next we estimate these integrals.

 If $|y| \leq N $ then by  (\ref{eca}), (\ref{ecb}) and
 (\ref{a56})  we have
 $$
 a_N (y) \lesssim \frac{\log N}{N}, \quad
 b_N (y) \lesssim \frac{1}{N^2}, \quad
A(\lambda, r) \lesssim \frac{1}{N} \;\; \forall r\geq 1.
 $$
 Therefore, from  (\ref{4.34}), (\ref{4.35}) and (\ref{e4})
 it follows that
 $$
\Phi_N (y)  \lesssim  \frac{1}{N}  \left ( \frac{\|q\|^2}{N} +
(\mathcal{E}_{\sqrt{N}} (q))^2  \right )+ (\mathcal{E}_{4N} (q))^2
 \frac{\log N}{N^2},
 $$
so we obtain
\begin{equation}
\label{I1a} I_1 \lesssim \frac{\|q\|^2}{N} + (\mathcal{E}_{\sqrt{N}}
(q))^2 +(\mathcal{E}_{4N} (q))^2
 \frac{\log N}{N}\lesssim \frac{\|q\|^2}{N} +
 (\mathcal{E}_{\sqrt{N}}(q))^2.
 \end{equation}

Next we estimate $I_2.$ If $\lambda = N^2+N+iy $ with $N \leq |y|
\leq N^2,$  then (\ref{eca}), (\ref{ecb}) and  (\ref{a56}) imply that
$$ a_N (y) \lesssim  \frac{1}{N} \log \left (1+\frac{N^2}{|y|} \right
), \quad b_N (y) \lesssim \frac{1}{N|y|}, \quad A(\lambda, r)
\lesssim N^{-\frac{1}{r}} |y|^{-1+\frac{1}{r}}.$$ Therefore, $$
A^{1/2} \left (\lambda,\frac{a}{2-a} \right) A^{1/2} \left (\lambda,
\frac{\beta}{2-\beta} \right ) \lesssim
N^{1-\frac{1}{a}-\frac{1}{\beta}}
|y|^{-2+\frac{1}{a}+\frac{1}{\beta}} =\frac{1}{N} N^\gamma
|y|^{-\gamma},
$$ where $$ 0< \gamma := 2-1/a - 1/\beta =
 1-1/a +1/b <1 $$ due to (\ref{ec100}).

 Now
from (\ref{4.34}),  (\ref{4.35}) and (\ref{e4}) we obtain $$ \Phi_N
(y) \lesssim N^\gamma |y|^{-1-\gamma} \left (\frac{\|q\|^2}{N} +
(\mathcal{E}_{\sqrt{N}} (q))^2 \right ) + (\mathcal{E}_{4N} (q))^2
\frac{|y|^{-\gamma}}{N^{2-\gamma}} \log \left (1+\frac{N^2}{|y|}
\right ), $$
 Since $\int_N^{N^2} y^{-1-\gamma} dy \lesssim
N^{-\gamma}$ and (with the change of variable $t=N^2/y $) $$
\int_N^{N^2} y^{-\gamma} \log (1+N^2/y)dy =N^{2-2\gamma} \int_0^N
\frac{1}{t^{2-\gamma}} \log(1+t) dt \lesssim N^{2-2\gamma}, $$ it
follows that
\begin{equation}
\label{I2a} I_2 \lesssim  \|q\|^2/N + (\mathcal{E}_{\sqrt{N}}
(q))^2 + N^{-\gamma} (\mathcal{E}_{4N} (q))^2.
 \end{equation}

Finally, we estimate $I_3.$ For
 $\lambda = N^2+N+iy $
with $ |y| \geq N^2$ we have by (\ref{eca}), (\ref{ecb}) and
(\ref{a56})  that $$ a_N (y) \lesssim \frac{1}{|y|^{1/2}}, \quad b_N
(y) \lesssim \frac{1}{|y|^{3/2}}, \quad A(\lambda, r) \lesssim
|y|^{-1+\frac{1}{2r}}. $$ Therefore, $$ A^{1/2} \left
(\lambda,\frac{a}{2-a} \right) A^{1/2} \left (\lambda,
\frac{\beta}{2-\beta} \right ) \lesssim
 |y|^{-\frac{3}{2}+\frac{1}{2a}+\frac{1}{2\beta}}
= |y|^{-\frac{1}{2}-\frac{\gamma}{2}}, $$ so (\ref{4.34}),
(\ref{4.35}) and (\ref{e4})  imply that $$ \Phi_N (y) \lesssim N^2
|y|^{-2 - \gamma/2} \left ( \frac{\|q\|^2}{N} +
(\mathcal{E}_{\sqrt{N}} (q))^2 \right ) + |y|^{-1 - \gamma/2}
(\mathcal{E}_{4N} (q))^2.
$$
Now, integrating over $|y| \geq N^2, $  we obtain
\begin{equation}
\label{I3a} I_3 \lesssim  \left ( \frac{\|q\|^2}{N} +
(\mathcal{E}_{\sqrt{N}} (q))^2 \right ) N^{-\gamma} +
(\mathcal{E}_{4N} (q))^2 N^{-\gamma}.
 \end{equation}
The estimates
 (\ref{I1a}), (\ref{I2a}) and (\ref{I3a}) yield
(\ref{ec101}), which completes the proof.

\end{proof}
\bigskip

\section{Appendix: Auxiliary Inequalities}

In this section we collect a few inequalities that justify crucial
steps in the proof of our main results and could be useful elsewhere.

The elementary inequality
\begin{equation}
\label{a1} (a+b)^\tau \leq 2^{\tau -1} (a^\tau + b^\tau ), \quad a,b
>0, \;\; \tau \geq 1,
\end{equation}
will be used throughout the text often without any specification. Of
course, (\ref{a1}) explains that for every fixed $\tau >0 $
$$ (a+b)^\tau \sim a^\tau + b^\tau,  \quad a,b >0.
$$

Next, for fixed $r \geq 1, $  we analyze the  behavior of the
function
\begin{equation}
\label{a2} A(z,r) = \left ( \sum_{k=0}^\infty \frac{1}{|z-k^2|^r}
\right )^{1/r}, \quad  r\geq 1, \;\;  z\in \mathbb{C}.
\end{equation}
We need estimates of this function and its integrals on properly
chosen contours in $\mathbb{C}.$ \vspace{3mm}

1. Horizontal lines $z=x+ih, \; x\in \mathbb{R}.$ \\
Since $A(x+ih,r)=A(x+i|h|,r), $ we assume for simplicity of the
writing that $h>0.$ If $z= x+ih $ then
\begin{equation}
\label{a3} \frac{1}{2}(|x-k^2|+h)\leq |z-k^2| \leq  |x-k^2|+h
\end{equation}
and
\begin{equation}
\label{a4} |z-k^2|^r \geq  2^{-r}(|x-k^2|+h)^r.
\end{equation}
Therefore,
\begin{equation}
\label{a5}  [A(x+ih,r)]^r \leq 2^r \sum_0^\infty
\frac{1}{(|x-k^2|+h)^r} =2^r ( \sigma_0^r + \sigma_1^r ),
\end{equation}
where
\begin{equation}
\label{a5a} \sigma_0^r=  \sum_0^{2N} \frac{1}{(|x-k^2|+h)^r} \leq
(2N+1) \cdot h^{-r}
\end{equation}
and
\begin{equation}
\label{a5b} \sigma_1^r = \sum_{2N+1}^\infty \frac{1}{(|x-k^2|+h)^r}.
\end{equation}

If
\begin{equation}
\label{a6} |x| \leq N^2 +N,
\end{equation}
and $k \geq 2N+1,$ then
\begin{equation}
\label{a7} |x-k^2| +h \geq  \frac{1}{2}(k^2 +h)\geq
\frac{1}{4}(k+\sqrt{h})^2.
\end{equation}
Indeed, if $x \leq 0,$ then (\ref{a7}) is obvious. If $0<x \leq
N^2+N, $ then $|x-k^2|  \geq k^2- (N^2+N) \geq  \frac{1}{2} k^2$
because $k \geq 2N+1.$

Therefore, if (\ref{a6}) holds, then
\begin{equation}
\label{a8} \sigma_1^r \leq \sum_{2N+1}^\infty
\frac{4^r}{(k+\sqrt{h})^{2r}} \leq \int_{2N}^\infty
\frac{4^r}{(\xi+\sqrt{h})^{2r}}d\xi = \frac{4^r}{2r-1} \left
(2N+\sqrt{h}\right )^{-2r+1}.
\end{equation}
Now, in view of (\ref{a5}), (\ref{a5a}) and (\ref{a8}), it follows
that
$$
A(x+ih,r) \leq 2(\sigma_0^r + \sigma_1^r)^{1/r} \lesssim
\frac{N^{1/r}}{h} + \frac{1}{(N^2 +h)^{1-1/(2r)}} \lesssim
\frac{1}{h} \left ( N^{1/r} +h^{\frac{1}{2r}}  \right ).
$$
The above inequalities imply the following.

\begin{Lemma}
\label{lem405} If $\; |x| \leq N^2 +N $ and $\; h \geq N^2,$ then
\begin{equation}
\label{a17} A(x \pm ih,r)\lesssim h^{\frac{1}{2r}-1}
\end{equation}
and
\begin{equation}
\label{a17a} \int_{-\omega}^{N^2+N} (A(x\pm ih,r))^m dx \lesssim
(N^2+\omega)\cdot  h^{-m(1-\frac{1}{2r})}, \quad m>0.
\end{equation}
If $m \geq 1,$ then
\begin{equation}
\label{a17b} \lim_{h\to \infty} \int_{-\omega}^{N^2+N} (A(x\pm
ih,r))^m dx =0.
\end{equation}
\end{Lemma}
\bigskip

2. Vertical lines $z= - \omega +iy, \; \omega >>1.$

In this case $ |z-k^2|\sim k^2 +\omega + |y| \sim \left
(k+\sqrt{\omega +|y|}\right)^2. $ Hence, for every fixed $r\geq 1$
we have
\begin{equation}
\label{a18}  A^r \sim  \sum_{k=0}^\infty \frac{1}{\left
(k+\sqrt{\omega +|y|}\right)^{2r}} \sim \int_0^\infty
 \frac{1}{\left
(x+\sqrt{\omega +|y|}\right)^{2r}} dx \lesssim
\frac{1}{(\omega+|y|)^{r-\frac{1}{2}}}.
\end{equation}
Therefore, the following holds.

\begin{Lemma}
\label{lem407} For fixed $r\geq 1, \; \omega >0$
\begin{equation}
\label{a19}  A(-\omega+iy, r)  \lesssim \left ( \frac{1}{\omega+|y|}
\right )^{1-\frac{1}{2r}}.
\end{equation}
Moreover, if $r>1$  then
\begin{equation}
\label{a20}  \int_{-\infty}^\infty A^2(-\omega+iy, r)dy \to 0 \quad
\text{as} \;\; \omega \to \infty,
\end{equation}
and if $r=1 $
\begin{equation}
\label{a21}  \int_{-\infty}^\infty A^3(-\omega+iy, 1)dy \to 0 \quad
\text{as} \;\; \omega \to \infty.
\end{equation}

\end{Lemma}
\bigskip

3. Now we analyze $A(z,r) $ on the vertical line
\begin{equation}
\label{a27}  \Lambda_N =\{\lambda = N^2+N+iy, \; y \in \mathbb{R}\}.
\end{equation}
Since $A(N^2+N+iy, r)=A(N^2+N+i|y|, r),$ we assume for simplicity
that $y\geq 0.$

One can easily see for $\lambda = N^2+N+iy $ that
\begin{equation}
\label{a28} |\lambda -k^2|\sim |N^2-k^2|+N+y.
\end{equation}
Therefore, in view of (\ref{a2})
\begin{equation}
\label{a29} A(\lambda,r) \sim \sigma_0 + \sigma_1   \quad \text{for}
\;\;  \lambda \in \Lambda_N,
\end{equation}
where
\begin{equation}
\label{a30} (\sigma_0)^r =\sum_{k=0}^{2N}
\frac{1}{(|N^2-k^2|+N+y)^r}, \quad (\sigma_1)^r=\sum_{k=2N+1}^\infty
\frac{1}{(|N^2-k^2|+N+y)^r}.
\end{equation}

Next we estimate $(\sigma_0)^r.$ Since $|N^2-k^2| =|N-k|(N+k) \sim
|N-k|N $ if $0 \leq k \leq 2N, $ it follows that $$ (\sigma_0)^r
\sim \sum_{j=0}^N \frac{1}{(Nj+N+y)^r} \sim \int_0^{N}
\frac{1}{(N\xi+N+y)^r} d\xi.
$$
 $$ \qquad \sim \begin{cases} \frac{1}{N}\log \left
(1+\frac{N^2}{N+y}\right )  & \text{if} \;\; r=1, \vspace{1mm}\\
\frac{1}{N} \left (\frac{1}{(N+y)^{r-1}}-\frac{1}{(N^2+N+y)^{r-1}}
\right )  & \text{if} \;\; r>1.
\end{cases}
$$ Therefore, by using Mean Value Theorem in the case $
r>1, \, y>N^2, $ we obtain the following.

\begin{Lemma}
\label{lem410} In the above notations,
\begin{equation}
\label{a37} \sigma_0 \sim  \frac{1}{N}\log \left
(1+\frac{N^2}{N+y}\right )  \quad \text{if} \;\; r=1,
\end{equation}
and for $r>1$
\begin{equation}
\label{a38} \sigma_0 \sim
\begin{cases} \frac{1}{N^{1/r}}
\cdot \frac{1}{(N+y)^{1-1/r}} & \text{if} \quad  0\leq y \leq N^2,
\vspace{1mm}\\  \frac{N^{1/r}}{y}  & \text{if} \quad y >N^2.
\end{cases}
\end{equation}
\end{Lemma}
\bigskip

4. Now we estimate the sum $ (\sigma_1)^r $ defined in (\ref{a30}).
If $k \geq 2N+1 $ then $$k^2 -N^2 +N +y \sim k^2 +y \sim (k
+\sqrt{y})^2,$$ so
$$
 (\sigma_1)^r \sim \sum_{2N+1}^\infty \frac{1}{(k
+\sqrt{y})^{2r}}
 \sim
\int_{2N}^\infty \frac{1}{(\xi +\sqrt{y})^{2r}} d\xi \sim
(2N+\sqrt{y})^{-2r+1}.
$$
Therefore, we obtain
\begin{equation}
\label{a53} \sigma_1 \sim (N^2+y)^{-1 +\frac{1}{2r}}.
\end{equation}
This, together with Lemma~\ref{lem410}, (\ref{a37}) and (\ref{a38}),
leads us to the following.

\begin{Lemma}
\label{lemA}  With notations (\ref{a2}), we have
\begin{equation}
\label{a55} A(N^2+N+iy,1)
 \sim \begin{cases}  \frac{\log N}{N}    &
\text{if} \quad   |y| \leq N,\\ \frac{1}{N} \log
(1+\frac{N^2}{|y|}) & \text{if} \quad   N \leq |y| \leq N^2,\\
\frac{1}{\sqrt{|y|}}   & \text{if} \quad  |y| \geq N^2,
\end{cases}
\end{equation}
and, for $r>1,$
\begin{equation}
\label{a56} A(N^2+N+iy,r) \sim \begin{cases}  \frac{1}{N} & \text{if}
\quad   |y| \leq N,\\ N^{-\frac{1}{r}} \,
|y|^{-1+\frac{1}{r}} & \text{if} \quad N \leq |y| \leq N^2,\\
|y|^{-1+\frac{1}{2r}}   & \text{if} \quad |y| \geq N^2.
\end{cases}
\end{equation}
\end{Lemma}

\begin{proof}
Indeed,  (\ref{a37}) and (\ref{a53}) lead to (\ref{a55}), and
(\ref{a38}) together with (\ref{a53}) imply (\ref{a56}).

\end{proof}

The inequality (\ref{a56}) helps us to give estimates of $
\int_{\Lambda_N}  A^2 (\lambda;r) dy $ from above. We have the
following three cases:

(i) $r>2; $

(i) $r=2; $

(iii) $ 1 < r <2. $

In either case ($r>1$),
\begin{equation}
\label{101} \int_{|y|\leq N } A^2 dy \lesssim  (1/N)^2 \cdot 2N
\lesssim \frac{1}{N}
\end{equation}
and
\begin{equation}
\label{102} \int_{|y|\geq N^2 } A^2 (\lambda, r) dy \lesssim
\int_{N^2}^\infty y^{-2+1/r} dy \lesssim N^{-2(1-1/r)}.
\end{equation}

The integration over the interval $[N, N^2]$ is also easy but the
result depends on $r$ in an essential way. By (\ref{a56}), the
middle line,
\begin{equation}
\label{103} \int_{N \leq |y|\leq N^2 } A^2  (\lambda, r) dy \lesssim
Y(N), \end{equation} where
$$
Y(N) := \int_N^{N^2} N^{-2/r} y^{-2+2/r} dy= \begin{cases} \frac{\log
N}{N} &  \text{if}  \;\; r=2;\\
\frac{r}{2-r}\left ( N^{-2(1-\frac{1}{r})} - N^{-1} \right )
&\text{if}  \;\; r\neq 2.
\end{cases}
$$
Therefore
\begin{equation}
\label{104}
 Y(N) \lesssim
\begin{cases}
 \frac{1}{N}    &  \text{if} \;\; r>2,\\
 \frac{ \log N}{N}  &  \text{if} \;\; r=2,\\
 N^{-2(1-\frac{1}{r})} &  \text{if} \;\; r<2.
\end{cases}
\end{equation}
Now the inequalities (\ref{101})--(\ref{104}) imply the following.

\begin{Corollary}
\label{cor17A} If $1< r < \infty, $ then
\begin{equation}
\label{108} \int_{\Lambda_N}  A^2 (\lambda;r) dy \leq
\begin{cases}
  \frac{1}{N}    &  \text{if} \;\; r>2,\\
 \frac{ \log N}{N}  &  \text{if} \;\; r=2,\\
 N^{-2(1-\frac{1}{r})} &  \text{if} \;\; r<2.
\end{cases}
\end{equation}
\end{Corollary}

Next we consider the case $r=1.$
\begin{Corollary}
\label{cor17B} From (\ref{a55}) it follows that
\begin{equation}
\label{110} \int_{\Lambda_N}  A^3 (\lambda;1) dy \lesssim
\frac{1}{N}.
\end{equation}
\end{Corollary}

\begin{proof}
The integral over $[0,N]$  does not exceed $C \frac{(\log N)^3}{N^2},
$ and the integral over $[N^2, \infty) $ is less than $C/N.$ Next,
with $w= 1+N^2/y, \; dy = - \frac{N^2}{(w-1)^2} dw $ we estimate
$\int_N^{N^2} A^3 (\lambda;1) dy $ by
\begin{equation}
\label{111a} \int_N^{N^2} \left (  \frac{1}{N} \log (1+N^2/y) \right
)^3 dy = \frac{1}{N} \int_2^{1+N} \frac{(\log w)^3}{(w-1)^2} dw
\lesssim \frac{1}{N}.
\end{equation}
These estimates lead immediately to (\ref{110}).
\end{proof}

Notice however that
\begin{equation}
\label{111} \int_{\Lambda_N} A^2 (\lambda,1) dy = \infty
\end{equation}
because by the third line of (\ref{a55}) we have $A(\lambda,1) \sim
1/y^{1/2} $ for $  y> N^2.$
\bigskip

5. We need a few estimates of double sums as well.

\begin{Lemma}
\label{lem105} The sequence
\begin{equation}
\label{105.1} A_N= \sum_{k=0}^N \sum_{m=N+1}^\infty \frac{1}{m^2
-k^2}, \quad N=0, 1, 2, \ldots
\end{equation}
is bounded by $A_0 = \pi^2/6.$  Moreover,  $A_N$ is monotone
decreasing, and $\; \;\lim A_N =\pi^2/8.$
\end{Lemma}

\begin{proof}
We have $\displaystyle A_0 = \sum_{m=1}^\infty \frac{1}{m^2} =
\pi^2/6 \;$ and
$$
A_{N+1} - A_N = \sum_{m=N+2}^\infty \frac{1}{m^2 -(N+1)^2} -
\sum_{k=0}^N \frac{1}{(N+1)^2 -k^2}
$$
$$
= \frac{1}{2(N+1)} \sum_{m=N+2}^\infty \left (\frac{1}{m-(N+1)}
-\frac{1}{m+(N+1)} \right ) $$
$$ - \frac{1}{2(N+1)} \sum_{k=0}^N
\left ( \frac{1}{N+1-k} +\frac{1}{N+1+k} \right )
$$
$$
=\frac{1}{2(N+1)} \left (\sum_{\nu=1}^{2(N+1)} \frac{1}{\nu}
-\sum_{\nu=1}^{N+1} \frac{1}{\nu}  -\sum_{\nu=N+1}^{2N+1}
\frac{1}{\nu} \right  ) = -\frac{1}{4(N+1)^2}.
$$
Therefore,
$$
A_N = A_0 + \sum_{n=1}^N (A_n-A_{n-1}) = \frac{\pi^2}{6}-
\frac{1}{4}\sum_{n=1}^N \frac{1}{n^2} \to \frac{\pi^2}{6}-\frac{1}{4}
\cdot \frac{\pi^2}{6} =\frac{\pi^2}{8}.
$$

\end{proof}

\begin{Lemma}
\label{lem106} Let $H \in \mathbb{N}, \; 0 < H < N,$ and let $$
\Delta_H = \{(k,m): \; 0\leq k \leq N, \; m\geq N+1, \; m-k \leq H
\}. $$ Then
\begin{equation}
\label{106.2} \sigma (N,H) = \sum_{\Delta_H} \frac{1}{m^2 -k^2} \leq
\frac{H}{N}.
\end{equation}
\end{Lemma}

\begin{proof}
Observe that $\Delta_H $ consists of points $(k,m)$ with integer
coordinates lying inside the triangle bounded by the lines $k=N, \;
m= N+1, m-k =H.$  Moreover,
\begin{equation}
\label{Delta} \Delta_H = \bigcup_{\nu=1}^H \ell_\nu,  \quad
\text{with} \quad \ell_\nu= \{(k,m)\in \Delta_H: \, m-k = \nu \},
\quad \#\ell_\nu =\nu,
\end{equation}
so  $ \# \Delta_H = H(H+1)/2. $ Since $ m \geq N+1, $ we have
$$
\sigma = \sum_{\Delta_H} \frac{1}{2m} \left (\frac{1}{m-k} +
\frac{1}{m+k} \right ) \leq \frac{1}{2(N+1)} \sum_{\Delta_H} \left
(\frac{1}{m-k} + \frac{1}{m+k} \right )
$$
$$
\leq \frac{1}{2(N+1)} \left (\sum_{\nu=1}^H \frac{1}{\nu} \#\ell_\nu
+  \frac{1}{N+1} \#\Delta_H \right ) = \frac{H}{2(N+1)} +
\frac{(H+1)H}{4(N+1)^2} \leq \frac{H}{N},
$$
which completes the proof.
\end{proof}

\end{document}